\documentclass[12pt,leqno,draft]{article}
\usepackage{amsfonts}
\usepackage{amsmath, amsthm, amsfonts, amssymb, color}
\usepackage{mathrsfs}
\usepackage{color}
\newtheorem{thm}{Theorem}[section]

\newtheorem{lem}[thm]{Lemma}

\newtheorem{rem}[thm]{Remark}
\theoremstyle{definition}

\newcommand{\scr}[1]{\mathscr #1}
\definecolor{wco}{rgb}{0.5,0.2,0.3}

\numberwithin{equation}{section} \theoremstyle{remark}

\newcommand{\ua}{\uparrow}
\newcommand{\da}{\downarrow}

\title{{\bf Regularities for distribution dependent SDEs with fractional noises}\footnote{
X. Fan is supported in part by the Natural Science Foundation of Anhui Province (No.
2008085MA10).
X. Huang is supported in part by the National Natural Science Foundation of China (No. 12271398).} }
\author{
{\bf   Xiliang Fan$^{a)}$,  Xing Huang$^{b)}$, Zewei Ling$^{a)}$,   }\\
 \footnotesize{ a)School of Mathematics and Statistics, Anhui Normal University, Wuhu 241002, China}\\
\footnotesize{ b)Center for Applied Mathematics, Tianjin
University, Tianjin 300072, China}\\
\footnotesize{fanxiliang0515@163.com,\ \ xinghuang@tju.edu.cn,\ \ ling122099@163.com}}
\begin{document}
\allowdisplaybreaks
\def\R{\mathbb R}  \def\ff{\frac} \def\ss{\sqrt} \def\B{\mathbf
B} \def\W{\mathbb W}
\def\N{\mathbb N} \def\kk{\kappa} \def\m{{\bf m}}
\def\ee{\varepsilon}\def\ddd{D^*}
\def\dd{\delta} \def\DD{\Delta} \def\vv{\varepsilon} \def\rr{\rho}
\def\<{\langle} \def\>{\rangle} \def\GG{\Gamma} \def\gg{\gamma}
  \def\nn{\nabla} \def\pp{\partial} \def\E{\mathbb E}
\def\d{\text{\rm{d}}} \def\bb{\beta} \def\aa{\alpha} \def\D{\scr D}
  \def\si{\sigma} \def\ess{\text{\rm{ess}}}
\def\beg{\begin} \def\beq{\begin{equation}}  \def\F{\scr F}
\def\Ric{\text{\rm{Ric}}} \def\Hess{\text{\rm{Hess}}}
\def\e{\text{\rm{e}}} \def\ua{\underline a} \def\OO{\Omega}  \def\oo{\omega}
 \def\tt{\tilde} \def\Ric{\text{\rm{Ric}}}
\def\cut{\text{\rm{cut}}} \def\P{\mathbb P} \def\ifn{I_n(f^{\bigotimes n})}
\def\C{\scr C}      \def\aaa{\mathbf{r}}     \def\r{r}
\def\gap{\text{\rm{gap}}} \def\prr{\pi_{{\bf m},\varrho}}  \def\r{\mathbf r}
\def\Z{\mathbb Z} \def\vrr{\varrho} \def\ll{\lambda}
\def\L{\scr L}\def\Tt{\tt} \def\TT{\tt}\def\II{\mathbb I}
\def\i{{\rm in}}\def\Sect{{\rm Sect}}  \def\H{\mathbb H}
\def\M{\scr M}\def\Q{\mathbb Q} \def\texto{\text{o}} \def\LL{\Lambda}
\def\Rank{{\rm Rank}} \def\B{\scr B} \def\i{{\rm i}} \def\HR{\hat{\R}^d}
\def\to{\rightarrow}\def\l{\ell}\def\iint{\int}
\def\EE{\scr E}\def\Cut{{\rm Cut}}
\def\A{\scr A} \def\Lip{{\rm Lip}}
\def\BB{\scr B}\def\Ent{{\rm Ent}}\def\L{\scr L}
\def\R{\mathbb R}  \def\ff{\frac} \def\ss{\sqrt} \def\B{\mathbf
B}
\def\N{\mathbb N} \def\kk{\kappa} \def\m{{\bf m}}
\def\dd{\delta} \def\DD{\Delta} \def\vv{\varepsilon} \def\rr{\rho}
\def\<{\langle} \def\>{\rangle} \def\GG{\Gamma} \def\gg{\gamma}
  \def\nn{\nabla} \def\pp{\partial} \def\E{\mathbb E}
\def\d{\text{\rm{d}}} \def\bb{\beta} \def\aa{\alpha} \def\D{\scr D}
  \def\si{\sigma} \def\ess{\text{\rm{ess}}}
\def\beg{\begin} \def\beq{\begin{equation}}  \def\F{\scr F}
\def\Ric{\text{\rm{Ric}}} \def\Hess{\text{\rm{Hess}}}
\def\e{\text{\rm{e}}} \def\ua{\underline a} \def\OO{\Omega}  \def\oo{\omega}
 \def\tt{\tilde} \def\Ric{\text{\rm{Ric}}}
\def\cut{\text{\rm{cut}}} \def\P{\mathbb P} \def\ifn{I_n(f^{\bigotimes n})}
\def\C{\scr C}      \def\aaa{\mathbf{r}}     \def\r{r}
\def\gap{\text{\rm{gap}}} \def\prr{\pi_{{\bf m},\varrho}}  \def\r{\mathbf r}
\def\Z{\mathbb Z} \def\vrr{\varrho} \def\ll{\lambda}
\def\L{\scr L}\def\Tt{\tt} \def\TT{\tt}\def\II{\mathbb I}
\def\i{{\rm in}}\def\Sect{{\rm Sect}}  \def\H{\mathbb H}
\def\M{\scr M}\def\Q{\mathbb Q} \def\texto{\text{o}} \def\LL{\Lambda}
\def\Rank{{\rm Rank}} \def\B{\scr B} \def\i{{\rm i}} \def\HR{\hat{\R}^d}
\def\to{\rightarrow}\def\l{\ell}
\def\8{\infty}\def\I{1}\def\U{\scr U}
\maketitle

\def\ti{\tilde}
\def\la{\lambda}
\def\th{\theta}
\def\sP{\mathscr{P}}
\def\sB{\mathscr {B}}
\def\sL{\mathscr{L}}
\def\de{\delta}
\def\ra{\rightarrow}
\def\sF{\mathscr{F}}
\def\cS{\mathcal{S}}
\def\ve{\varepsilon}
\def\na{\nabla}
\def\be{\beta}
\def\cH{\mathcal{H}}
 \def\sC{\mathscr {C}}

\begin{abstract}
In this paper, we investigate the regularities for a class of distribution dependent SDEs driven by two independent fractional noises $B^H$ and $\ti B^{\ti H}$ with Hurst parameters $H\in(0,1)$ and $\ti H\in(1/2,1)$.
We establish the log-Harnack inequalities and Bismut formulas for the Lions derivative to this type of equations with distribution dependent noise, in
both non-degenerate and degenerate cases.
Our proofs consist of utilizing coupling arguments which are indeed backward couplings introduced by F.-Y. Wang \cite{Wang12b},
together with a careful analysis of fractional derivative operator.

\end{abstract} \noindent
 \vskip 2cm

 \section{Introduction}
Let $\sP(\R^d)$ be the space of all probability measures on $\R^d$ equipped with the weak topology. For any $\theta\geq 1$, set $\scr P_\theta(\R^d):=\{\mu\in\sP(\R^d):\mu(|\cdot|^\th)^{1/\th}\}<\infty$.
In this article, we are concerned with a distribution dependent stochastic differential equations (DDSDEs) of the form
\beq\label{In1}
\d X_t= b_t(X_t, \L_{X_t})\d t+\si_t\d B^{H}_t+\ti\si_t(\L_{X_t})\d\ti B^{\ti H}_t,\ \ X_0=\xi,
\end{equation}
where $\sL_{X_t}$ denotes the law of $X_t$, $b:[0,T]\times\R^d\times\scr P(\R^d)\rightarrow\R^d, \si:[0,T]\rightarrow\R^d\otimes\R^d, \ti\si:[0,T]\times\scr P(\R^d)\rightarrow\R^d\otimes\R^d, \xi$ is an $\R^d$-valued random variable,
and $B^H,\ti B^{\ti H}$ are respectively two independent fractional Brownian motions (FBMs) with Hurst parameters
$H\in(0,1)$ and $\ti H\in(1/2,1)$ independent of $\xi$.

The FBM is commonly viewed as the simplest stochastic process modelling time correlated noise.
A $d$-dimensional FBM $(B^H_t)_{t\in[0,T]}=(B^{H,1}_t,\cdots,B^{H,d}_t)_{t\in[0,T]}$ with Hurst parameter $H\in(0,1)$ is a centered, $H$-self similar Gaussian process with the covariance function $\E(B^{H,i}_tB^{H,j}_s)=R_H(t,s)\delta_{i,j}$, where
\beq\label{In3}
 R_H(t,s):=\frac{1}{2}\left(t^{2H}+s^{2H}-|t-s|^{2H}\right),\ \  t,s\in[0,T].
\end{equation}
This implies that the FBM generalizes the standard Wiener process ($H=1/2$) and has stationary increments.
However, the increments are correlated with a power law correlation decay,
which asserts the FBM is a non-Markovian process that is the dominant feature of equation \eqref{In1}.
This means that the techniques based on the It\^{o} calculus are not applicable and then substantial new difficulties will appear in this setting.

DDSDE is also called McKean-Vlasov or mean-field SDE, which was first introduced in the pioneering work \cite{McKean66} to model plasma dynamics.
The importance of DDSDEs is due to their description of limiting behaviours of individual particles which interact with each other in a mean-field sense,
when the number of particles tends to infinity.
Another important feature of DDSDEs is their intrinsic link with nonlinear Fokker-Planck equations that characterize the evolution of the marginal laws of DDSDEs.
For these reasons, DDSDEs appear widely in applications, including fluid dynamics, mean-field games, biology and mathematical finance etc,
and then have received increasing attentions, see \cite{BT97,CD13,JW17,LL07} and the references therein.
Recently, DDSDEs have been applied in \cite{BLPR17,CM18,Li18a,RW21} to study smoothness of associated PDE
which involves the Lions derivative introduced by P.-L. Lions in his lectures \cite{Cardaliaguet13}.
Moreover, Bismut formula for the Lions derivative, Harnack type inequality, gradient estimate and exponential ergodicity have been studied,
see for instance \cite{BRW,RW,Song,Wang18}.

Contrary to the previously mentioned works, here we aim to investigate the regularities of equation \eqref{In1} perturbed by two independent fractional noises.
That is, we study the regularity of the maps
\beg{align*}
\mu\mapsto P_t^*\mu,\ \  t\in[0,T],
\end{align*}
where $P_t^*\mu:=\sL_{X_t}$ for $X_t$ solving \eqref{In1} with initial distribution $\sL_{X_0}=\mu\in\sP_p(\R^d)$.
Observe that a probability measure is determined by integrals of $f\in\sB_b(\R^d)$, the collection of all bounded measurable functions on $\R^d$,
it suffices to investigate the regularity of the functionals
\beg{align*}
\mu\mapsto(P_tf)(\mu):=\int_{\R^d}f\d(P_t^*\mu) ,\ \ f\in\sB_b(\R^d), t\in[0,T].
\end{align*}
More precisely, with regards to equation \eqref{In1}, we address the following question:

(Q) Under what conditions does the functional $P_tf$ have dimensional-free Harnack inequalities and Bismut formulas?

Our main reasons for doing so are the following.

(i) As pointed out in \cite{Banos18} which investigated the sensitivity of prices of
options with respect to the initial value of the underlying asset price, the Bismut formula gives a better approximation of the sensitivity.
In addition, the Harnack inequality may imply the gradient estimate and entropy estimate.

(ii) It was shown in our previous work \cite{FHSY} that for distribution-free noise ($\ti\sigma=0$ in equation \eqref{In1}),
Bismut formulas for $P_tf$ are established by using Malliavin calculus.
However, for distribution dependent noise, these formulas are still open due to technical difficulty (see the reason at the beginning of Section 4 in \cite{FHSY}).

In contrast with Brownian motion case, DDSDEs driven by FBM have been much less studied.
In addition to \cite{FHSY} mentioned above, we also established the large and moderate deviation principles for DDSDE driven by a FBM ($\sigma=0$ in equation \eqref{In1}).
See also the article \cite{GHM2} for the well-posedness result to DDSDE driven by a FBM ($\sigma=1,\ti\sigma=0$ in equation \eqref{In1})
with irregular, possibly distributional drift via some stability estimates.
To our best knowledge, none of the questions we ask here for equation \eqref{In1} with distribution-dependent and possibly degenerate noise
have been addressed so far.
It appears that they require a novel set of tools and ideas.
Our strategy in this paper builds on the work of the second author and Wang \cite{HW22}, which handled DDSDEs driven by a standard Brownian motion
\beq\label{In2}
\d X_t= b_t(X_t, \L_{X_t})\d t+\si_t(\L_{X_t})\d B^{\ff 1 2}_t,\ \ t\in[0,T].
\end{equation}
Therein the authors introduced a noise decomposition argument to the equation, which allows to obtain the Harnack inequality, Bismut formula and exponential ergodicity for equation \eqref{In2}.
In this paper, we first show the well-posedness of equation \eqref{In1}.
Then, instead of appealing to Malliavin calculus, we establish the log-Harnack inequalities and Bismut formula for the Lions derivative
to equation \eqref{In1} in  both non-degenerate and degenerate cases,
in which our proofs are based entirely on a combination of coupling argument and a careful analysis of fractional derivative operator.
Let us stress here that in our proofs, the invertible condition imposed on $\sigma$ is essential, and the constructed couplings are indeed backward couplings
which were first introduced in \cite{Wang12b}.

We conclude this introduction with the structure of the paper.
In Section 2, we recall some well-known facts on fractional calculus, FBM and the Lions derivative.
Section 3 contains the well-posedness result of DDSDE driven by FBM.
In Section 4, we state and prove our main results concerning the regularities for DDSDEs with distribution-dependent and possibly degenerate fractional noise.

\section{Preliminaries}

In this section, we recall some basic elements of fractional calculus, fractional Brownian motion and the Lions derivative.

\subsection{Fractional calculus}

Let $a,b\in\R$ with $a<b$.
For $f\in L^1([a,b],\R)$ and $\alpha>0$, the left-sided (respectively right-sided) fractional Riemann-Liouville integral of $f$ of order $\alpha$
on $[a,b]$ is defined as
\beg{align}\label{FrIn}
&I_{a+}^\alpha f(x)=\frac{1}{\Gamma(\alpha)}\int_a^x\frac{f(y)}{(x-y)^{1-\alpha}}\d y\\
&\qquad\left(\mbox{respectively}\ \ I_{b-}^\alpha f(x)=\frac{(-1)^{-\alpha}}{\Gamma(\alpha)}\int_x^b\frac{f(y)}{(y-x)^{1-\alpha}}\d y\right).\nonumber
\end{align}
Here $x\in(a,b)$ a.e., $(-1)^{-\alpha}=\e^{-i\alpha\pi}$ and $\Gamma$ stands for the Gamma function.
In particular, when $\alpha=n\in\N$, they are consistent with the usual $n$-order iterated integrals.

Fractional differentiation can be defined as an inverse operation.
Let $\alpha\in(0,1)$ and $p\geq1$.
If $f\in I_{a+}^\alpha(L^p([a,b],\R))$ (respectively $I_{b-}^\alpha(L^p([a,b],\R)))$, then the function $g$ satisfying $I_{a+}^\alpha g=f$ (respectively $I_{b-}^\alpha g=f$) is unique in $L^p([a,b],\R)$ and it coincides with the left-sided (respectively right-sided) Riemann-Liouville derivative
of $f$ of order $\alpha$ given by
\beg{align*}
&D_{a+}^\alpha f(x)=\frac{1}{\Gamma(1-\alpha)}\frac{\d}{\d x}\int_a^x\frac{f(y)}{(x-y)^\alpha}\d y\\
&\qquad\left(\mbox{respectively}\ D_{b-}^\alpha f(x)=\frac{(-1)^{1+\alpha}}{\Gamma(1-\alpha)}\frac{\d}{\d x}\int_x^b\frac{f(y)}{(y-x)^\alpha}\d y\right).
\end{align*}
The corresponding Weyl representation is of the form
\beg{align}\label{FrDe}
&D_{a+}^\alpha f(x)=\frac{1}{\Gamma(1-\alpha)}\left(\frac{f(x)}{(x-a)^\alpha}+\alpha\int_a^x\frac{f(x)-f(y)}{(x-y)^{\alpha+1}}\d y\right)\\
&\qquad\left(\mbox{respectively}\ \ D_{b-}^\alpha f(x)=\frac{(-1)^\alpha}{\Gamma(1-\alpha)}\left(\frac{f(x)}{(b-x)^\alpha}+\alpha\int_x^b\frac{f(x)-f(y)}{(y-x)^{\alpha+1}}\d y\right)\right),\nonumber
\end{align}
where the convergence of the integrals at the singularity $y=x$ holds pointwise for almost all $x$ if $p=1$ and in the $L^p$ sense if $p>1$.
For in-depth treatments, we refer the reader to \cite{SK}.

\subsection{Fractional Brownian motion}

Let $(\Omega,\sF,\P)$ be a probability space carrying a $d$-dimensional FBM $B^H$ with Hurst parameter $H\in(0,1)$ on the interval $[0,T]$.
We suppose that there is a sufficiently rich sub-$\si$-algebra $\sF_0\subset\sF$ independent of $B^H$ such that
for any $\mu\in\sP_p(\R^d)$ there exists a random variable $\xi\in L^p(\Omega\ra\R^d,\sF_0,\P)$ with distribution $\mu$.
Let $\{\sF_t\}_{t\in[0,T]}$ be the filtration generated by $B^H$, completed and augmented by $\sF_0$.

Let $\mathscr{E}$ be the set of step functions on $[0,T]$ and $\mathcal {H}$ the Hilbert space defined as the closure of
$\mathscr{E}$ with respect to the scalar product
\beg{align*}
\left\langle (\mathbb I_{[0,t_1]},\cdot\cdot\cdot,\mathbb I_{[0,t_d]}),(\mathbb I_{[0,s_1]},\cdot\cdot\cdot,\mathbb I_{[0,s_d]})\right\rangle_\cH=\sum\limits_{i=1}^dR_H(t_i,s_i).
\end{align*}
Recall here that $R_H(\cdot,\cdot)$ is given in \eqref{In3}.
The mapping $(\mathbb I_{[0,t_1]},\cdot\cdot\cdot,\mathbb I_{[0,t_d]})\mapsto\sum_{i=1}^dB_{t_i}^{H,i}$ can be extended to an isometry between $\cH$ (also called the reproducing kernel Hilbert space) and the Gaussian space $\mathcal {H}_1$ associated with $B^H$.
Denote this isometry by $\psi\mapsto B^H(\psi)$.
Besides, by \cite{DU98} we know that $R_H(t,s)$ has an integral representation of the form
\beg{align*}
 R_H(t,s)=\int_0^{t\wedge s}K_H(t,r)K_H(s,r)\d r,
\end{align*}
where $K_H$ is a square integrable kernel defined by
\beg{align*}
K_H(t,s)=\Gamma\left(H+\frac{1}{2}\right)^{-1}(t-s)^{H-\frac{1}{2}}F\left(H-\frac{1}{2},\frac{1}{2}-H,H+\frac{1}{2},1-\frac{t}{s}\right),
\end{align*}
in which $F(\cdot,\cdot,\cdot,\cdot)$ is the Gauss hypergeometric function.
See, e.g., \cite{DU98,Nikiforov&Uvarov88} for further details.

Next, we define the linear operator $K_H^*:\mathscr{E}\rightarrow L^2([0,T],\R^d)$ as follows
\beg{align*}
(K_H^*\psi)(s)=K_H(T,s)\psi(s)+\int_s^T(\psi(r)-\psi(s))\frac{\partial K_H}{\partial r}(r,s)\d r.
\end{align*}
Owing to \cite{Alos&Mazet&Nualart01a},
the relation $\langle K_H^*\psi,K_H^*\phi\rangle_{L^2([0,T],\R^d)}=\langle\psi,\phi\rangle_\mathcal {H}$ holds for all $\psi,\phi\in\mathscr{E}$,
and then by the bounded linear transform theorem, $K_H^*$ can be extended to an isometry between $\mathcal{H}$ and $L^2([0,T],\R^d)$.
As a consequence, by \cite{Alos&Mazet&Nualart01a} again,
there exists a $d$-dimensional Wiener process $W$ defined on $(\Omega,\sF,\P)$ such that $B^H$ has the following Volterra-type representation
\beg{align}\label{IRFor}
B_t^H=\int_0^tK_H(t,s)\d W_s, \ \ t\in[0,T].
\end{align}
In addition, define the operator $K_H: L^2([0,T],\mathbb{R}^d)\rightarrow I_{0+}^{H+1/2 }(L^2([0,T],\mathbb{R}^d))$ by
\beg{align*}
 (K_H f)(t)=\int_0^tK_H(t,s)f(s)\d s.
\end{align*}
According to \cite{DU98},
we obtain that $K_H$ is an isomorphism and for any $f\in L^2([0,T],\mathbb{R}^d)$,
\begin{equation}\nonumber
(K_H f)(s)=
\left\{
\begin{array}{ll}\vspace{0.3cm}
I_{0+}^{1}s^{H-1/2}I_{0+}^{H-1/2}s^{1/2-H}f,\ \ H\in(1/2,1),\\
I_{0+}^{2H}s^{1/2-H}I_{0+}^{1/2-H}s^{H-1/2}f,\ \ H\in(0,1/2).
\end{array} \right.
\end{equation}
Then for each $h\in I_{0+}^{H+1/2}(L^2([0,T],\R^d))$, the inverse operator $K_H^{-1}$ is of the form
\begin{equation}\label{InOp}
(K_H^{-1}h)(s)=
\left\{
\begin{array}{ll}\vspace{0.3cm}
s^{H-1/2}D_{0+}^{H-1/2}s^{1/2-H}h',\ \ \ \ \ \ \ H\in(1/2,1),\\
s^{1/2-H}D_{0+}^{1/2-H}s^{H-1/2}D_{0+}^{2H}h,\ \ H\in(0,1/2).
\end{array} \right.
\end{equation}
In particular, when $h$ is absolutely continuous, we get
\beg{align}\label{-HOF}
(K_H^{-1}h)(s)=s^{H-\frac{1}{2}}I_{0+}^{\frac{1}{2}-H}s^{\frac{1}{2}-H}h',\ \ H\in(0,1/2).
\end{align}

\subsection{The Lions derivative}

For $p>1$, define the $L^p$-Wasserstein distance on $\sP_p(\R^d)$ as follows
\beg{align*}
\mathbb{W}_p(\mu,\nu):=\inf_{\pi\in\sC(\mu,\nu)}\left(\int_{\R^d\times\R^d}|x-y|^p\pi(\d x, \d y)\right)^\ff 1 p,\ \ \mu,\nu\in\sP_p(\R^d),
\end{align*}
where $\sC(\mu,\nu)$ is the set of all couplings of $\mu$ and $\nu$.
It is well-known that $(\sP_p(\R^d),\mathbb{W}_p)$ is a Polish space.
Throughout this paper, denote $|\cdot|$ and $\<\cdot,\cdot\>$ respectively for the Euclidean norm and inner product,
and for a matrix, denote by $\|\cdot\|$ the operator norm.
$\|\cdot\|_{L^p_\mu}$ stands for the norm of the space $ L^p(\R^d\ra\R^d,\mu)$ and for a random variable $\eta$, $\sL_\eta$ denotes its distribution.

Now, we present the definition of the Lions derivative, see, e.g., \cite{BRW,Cardaliaguet13,{HW21}} for further details.

\beg{defn} Let $p\in(1,\infty)$.
\beg{enumerate}
\item[(1)]  A continuous function $f$ on $\sP_p(\R^d)$ is called intrinsically differentiable, if for any $\mu\in\sP_p(\R^d)$.
\beg{align*}
L^p(\R^d\ra\R^d,\mu)\ni\phi\mapsto D^L_\phi f(\mu):=\lim_{\ve\da0}\ff{f(\mu\circ(\mathrm{Id}+\ve\phi)^{-1})-f(\mu)}\ve\in\R
\end{align*}
is a well defined bounded linear operator.
In this case, the norm of the intrinsic derivative $D^Lf(\mu)$ is given by
\beg{align*}
\|D^Lf(\mu)\|_{L^{p^*}_\mu}=\sup_{\|\phi\|_{L^p_\mu}\leq 1}|D^L_\phi f(\mu)|,
\end{align*}
where $p^*=\ff {p-1}p$.

\item[(2)] $f$ is called $L$-differentiable on $\sP_p(\R^d)$, if it is intrinsically differentiable and
\beg{align*}
\lim_{\|\phi\|_{L^p_\mu}\ra0}\ff{|f(\mu\circ(\mathrm{Id}+\phi)^{-1})-f(\mu)-D^L_\phi f(\mu)|}{\|\phi\|_{L^p_\mu}}=0,\ \ \mu\in\sP_p(\R^d).
\end{align*}
If $f$ is $L$-differentiable on $\sP_p(\R^d)$ such that $D^Lf(\mu)(x)$ has a jointly continuous version in $(\mu,x)\in\sP_p(\R^d)\times\R^d$,
we denote $f\in C^{(1,0)}(\sP_p(\R^d))$.

\item[(3)] $g$ is called differentiable on $\R^d\times\sP_p(\R^d)$, if for any $(x,\mu)\in\R^d\times\sP_p(\R^d)$,
$g(\cdot,\mu)$ is differentiable and $g(x, \cdot)$ is $L$-differentiable.
Moreover, if $D^Lg(x,\cdot)(\mu)(y)$ and $\na g(\cdot,\mu)(x)$ are jointly continuous in $(x,y,\mu)\in\R^d\times\R^d\times\sP_p(\R^d)$,
we denote $g\in C^{1,(1,0)}(\R^d\times\sP_p(\R^d)$.

\end{enumerate}
\end{defn}

For a vector-valued function $f=(f_i)$ or a matrix-valued function $f=(f_{ij})$ with $L$-differentiable components, we simply write
\beg{align*}
D^Lf(\mu)=(D^Lf_i(\mu))  \ \  \mathrm{or}\ \ D^Lf(\mu)=(D^Lf_{ij}(\mu)).
\end{align*}
Let us finish this part by giving a formula for the $L$-derivative that are needed later on.

\beg{lem}\label{FoLD}(\cite[Theorem 2.1]{BRW})
Let $(\Omega,\sF,\P)$ be an atomless probability space and $\xi,\eta\in L^p(\Omega\ra\R^d,\P)$ with $p\in(1,\infty)$.
If $f\in C^{(1,0)}(\sP_p(\R^d))$, then
\beg{align*}
\lim_{\ve\da0}\ff {f(\sL_{\xi+\ve\eta})-f(\sL_\xi)} \ve=\E\<D^Lf(\sL_\xi)(\xi),\eta\>.
\end{align*}
\end{lem}

\section{Well-posedness of DDSDE by fractional noises}

In this section, we fix $H\in(0,1),\ti H\in(1/2,1)$ and consider the following DDSDE driven by fractional Brownian motions:
\beq\label{GeEq}
\d X_t= b_t(X_t, \L_{X_t})\d t+\si_t\d B^{H}_t+\ti\si_t(\L_{X_t})\d\ti B^{\ti H}_t,\ \ X_0=\xi,
\end{equation}
where the coefficients $b:[0,T]\times\R^d\times\scr P(\R^d)\rightarrow\R^d, \si:[0,T]\rightarrow\R^d\otimes\R^d,
\ti\si:[0,T]\times\scr P(\R^d)\rightarrow\R^d\otimes\R^d$ are measurable functions,
$(B^H_t)_{t\in[0,T]}$ and $(\ti B^{\ti H}_t)_{t\in[0,T]}$ are two independent fractional Brownian motions with Hurst parameters $H$ and $\ti H$, respectively,
and $\xi\in L^p(\Omega\ra\R^d,\sF_0,\P)$ with $p\geq1$.
To show the well-posedness of \eqref{GeEq}, we introduce the following hypothesis.
\beg{enumerate}
\item[\textsc{\textbf{(H1)}}] There exists a non-decreasing function $\kappa_\cdot$ such that for every $t\in[0,T], x,y\in\R^d, \mu,\nu\in\sP_p(\R^d)$,
\beg{align*}
|b_t(x,\mu)-b_t(y,\nu)|\le \kappa_t(|x-y|+\W_p(\mu,\nu)),\ \ \|\ti\si_t(\mu)-\ti\si_t(\nu)\|\le \kappa_t\W_p(\mu,\nu),
\end{align*}
and
\beg{align*}
|b_t(0,\de_0)|+\|\si_t\|+\|\ti\si_t(\de_0)\|\le \kappa_t.
\end{align*}
\end{enumerate}
Now, for every $p\geq1$, let $\cS^p([0,T])$ be the space of  $\R^d$-valued, continuous $(\sF_t)_{t\in[0,T]}$-adapted processes $\psi$ on $[0,T]$ such that
$$\|\psi\|_{\cS^p}:=\bigg(\E\sup_{t\in[0,T]}|\psi_t|^p\bigg)^{1/p}<\8,$$
and let the letter $C$ with or without indices denote generic constants, whose values may change from line to line.

\beg{defn}\label{DefS}
A stochastic process $X=(X_t)_{0\leq t \leq T}$ on $\R^d$ is called a solution of \eqref{GeEq},
if $X\in\cS^p([0,T])$ and $\P$-a.s.,
$$X_t=\xi+\int_0^tb_s(X_s,\sL_{X_s})\d s+\int_0^t\si_s\d B_s^H+\int_0^t\ti\si_s(\L_{X_s})\d\ti B^{\ti H}_s,\ \ t\in[0,T].$$
\end{defn}

\beg{rem}\label{RemS}
Note that $\si_\cdot$ and $\ti\si_\cdot(\sL_{X_\cdot})$ are both deterministic functions,
then $\int_0^t\si_s\d B_s^H$ and $\int_0^t\ti\si_s(\sL_{X_s})\d\ti B_s^{\ti H}$ can be regarded as Wiener integrals with respect to fractional Brownian motions.
\end{rem}

\beg{thm}\label{WP}
Suppose that $\xi\in L^p(\Omega\ra\R^d,\sF_0,\P)$ with $p\geq1$ and one of the following conditions:
\beg{enumerate}
\item[(I)] $H\in(1/2,1)$, $b,\si,\ti\si$ satisfy \textsc{\textbf{(H1)}} and $p>\max\{1/H,1/\ti H\}$;
\item[(II)] $H\in(0,1/2), b,\ti\si$ satisfies \textsc{\textbf{(H1)}}, $\si_t$ does not depend on $t$ and $p>1/\ti H$.
\end{enumerate}
Then Eq. \eqref{GeEq} has a unique solution $X\in\cS^p([0,T])$. Moreover, let $(X_t^\mu)_{t\in[0,T]}$ be the solution to \eqref{GeEq} with $\sL_{X_0}=\mu\in \scr P_p(\R^d)$
and denote $P^*_t\mu=\sL_{X_t^\mu}, t\in[0,T]$. Then it holds
\beg{align}\label{1Est}
\W_p(P^*_t\mu,P^*_t\nu)\leq C_{p,T,\kappa,\ti H}\W_p(\mu,\nu),\ \ \mu,\nu\in\sP_p(\R^d).
\end{align}
\end{thm}
\begin{proof}
Since the case of $H\in(0,1/2)$ is easier, we only handle the case of $H\in(1/2,1)$ and provide a sketch.
For any $\mu\in C([0,T],\scr P_{p})$, consider
\beg{align*}
\d X_t=b_t(X_t,\mu_t)\d t+\si_t\d B_t^H+\ti\si_t(\mu_t)\d\ti B_t^{\ti H},\ \ t\in[0,T], X_0=\xi.
\end{align*}
Denote its solution as $X_t^{\mu,\xi}$. We first assert that $\E\Big(\sup_{t\in[0,T]}|X_t^{\mu,\xi}|^p\Big)<\infty$ with $p>\max\{1/H,1/\ti H\}$.
Indeed, by \textsc{\textbf{(H1)}} and the H\"{o}lder inequality, we deduce that
\beg{align}\label{Pf1-WP}
&\E\bigg(\sup\limits_{t\in[0,T]}|X_t^{\mu,\xi}|^p\bigg)\cr
\leq& 4^{p-1}\E|\xi|^p+12^{p-1}(T\kappa_T)^p\bigg(1+\sup_{t\in[0,T]}\mu_t(|\cdot|^p)+\ff 1 T\E \int_0^T\sup_{s\in[0,t]}|X_s^{\mu,\xi}|^p\d t\bigg)\nonumber\\
& +4^{p-1}\E\left(\sup\limits_{t\in[0,T]}\left|\int_0^t\si_s\d B_s^H\right|^p\right)+4^{p-1}\E\left(\sup\limits_{t\in[0,T]}\left|\int_0^t\ti\si_s(\mu_s)\d\ti B_s^{\ti H}\right|^p\right).
\end{align}
By a similar analysis of \cite[Step 1]{FHSY}, we derive that for any $p>\max\{1/H,1/\ti H\}$,
\beg{align*}
&\E\left(\sup\limits_{t\in[0,T]}\left|\int_0^t\si_s\d B_s^H\right|^p\right)+\E\left(\sup\limits_{t\in[0,T]}\left|\int_0^t\ti\si_s(\mu_s)\d\ti B_s^{\ti H}\right|^p\right)\cr
\leq& C_{p,T,\kappa,H,\ti H}\left(1+\bigg(\sup_{t\in[0,T]}\mu_t(|\cdot|^p)\bigg)\right).
\end{align*}
This, together with \eqref{Pf1-WP}, implies $\E\Big(\sup_{t\in[0,T]}|X_t^{\mu,\xi}|^p\Big)<\infty$.

Now, define the mapping $\Phi^{\xi}:C([0,T],\scr P_p(\R^d))\to C([0,T],\scr P_p(\R^d))$ as
$$\Phi_t^{\xi}(\mu)=\L_{X_t^{\mu,\xi}},\ \ t\in[0,T].$$
By \textsc{\textbf{(H1)}}, we have
\beg{align*}
\E|X_t^{\mu,\xi}-X_t^{\nu,\tilde{\xi}}|^p\leq&3^{p-1}\E|\xi-\tilde{\xi}|^p +3^{p-1}\E\left|\int_0^t(b_s(X_s^{\mu,\xi},\mu_s)-b_s(X_s^{\nu,\tilde{\xi}},\nu_s))\d s\right|^p\cr
&+3^{p-1}\E\left|\int_0^t(\ti\si_s(\mu_s)-\ti\si_s(\nu_s))\d\ti B_s^{\ti H}\right|^p\cr
\leq&3^{p-1}\E|\xi-\tilde{\xi}|^p
+(6t)^{p-1}\kappa^p_t\int_0^t\left(\E|X_s^{\mu,\xi}-X_s^{\nu,\tilde{\xi}}|^p+\W_p(\mu_s,\nu_s)^p\right)\d s\cr
&+3^{p-1}C_{p,\ti H}\kappa^p_tt^{p\ti H-1}\int_0^t\W_p(\mu_s,\nu_s)^p\d s,
\end{align*}
where we use \cite[(3.8) in the proof of Theorem 3.1]{FHSY} and $p>1/{\ti H}$ in the last inequality.\\
Then, we get
\beg{align*}
\E|X_t^{\mu,\xi}-X_t^{\nu,\tilde{\xi}}|^p\leq&3^{p-1}\E|\xi-\tilde{\xi}|^p+C_{p,T,\kappa,\ti H}\int_0^t\E|X_s^{\mu,\xi}-X_s^{\nu,\tilde{\xi}}|^p\d s\cr
&+C_{p,T,\kappa,\ti H}\int_0^t\W_p(\mu_s,\nu_s)^p\d s,
\end{align*}
which, together with the Gronwall lemma, implies
\beg{align}\label{2PfEst}
\E|X_t^{\mu,\xi}-X_t^{\nu,\tilde{\xi}}|^p\leq C_{p,T,\kappa,\ti H}\E|\xi-\tilde{\xi}|^p+C_{p,T,\kappa,\ti H}\int_0^t\W_p(\mu_s,\nu_s)^p\d s.
\end{align}
Therefore, for any $\lambda>0$, we have
  \begin{align*}
\sup_{t\in[0,T]}\e^{-\lambda p t}\W_{p}(\Phi^{\xi}_t(\mu),\Phi^{\xi}_t(\nu))^p\leq \ff{C_{p,T,\kappa,\ti H}}{\lambda}\sup_{t\in[0,T]}\e^{-\lambda p t}\W_{p}(\mu_t, \nu_t)^p.
\end{align*}
Take $\lambda_0$ satisfying $\ff{C_{p,T,\kappa,\ti H}}{\lambda}<\frac{1}{2^p}$ and
let
$E^{\xi}:= \{\mu\in C([0,T];\scr P_p(\R^d)):\mu_0=\L_{\xi}\}$ equipped with the complete metric
$$\rr_{\lambda_0}(\nu,\mu):= \sup_{t\in [0,T]}\e^{-\lambda_0t} \W_{p}(\nu_t,\mu_t),\ \ \mu,\nu\in E^{\xi}.$$
Hence, it holds
  \begin{align*}
\rho_{\lambda_0}(\Phi^{\xi}(\mu),\Phi^{\xi}(\nu))< \frac{1}{2}\rho_{\lambda_0}(\mu,\nu),\ \ \mu,\nu\in E^{\xi}.
\end{align*}
Using the Banach fixed point theorem, we conclude that
\begin{equation*} \Phi^{\xi}_t(\mu)= \mu_t,\ \ t\in [0,T]\end{equation*}  has a unique solution $\mu\in E^{\xi},$ which means that \eqref{GeEq} has a unique strong solution on $[0,T]$ with initial value $\xi$.

Next, applying \eqref{2PfEst} for $\mu_t=P^*_t\mu, \nu_t=P^*_t\nu$, and taking $\xi,\tilde{\xi}$ satisfying $\L_{\xi}=\mu,\L_{\tilde{\xi}}=\nu$ and
$\E|\xi-\tilde{\xi}|^p=\W_{p}(\mu,\nu)^p$, there exists a constant $C_{p,T,\kappa,\ti H}>0$ such that
\begin{align*}
&\sup_{s\in[0,t]}\W_{p}(P^*_s\mu,P^*_s\nu)^p \leq C_{p,T,\kappa,\ti H}\W_{p}(\mu,\nu)^p+
C_{p,T,\kappa,\ti H}\int_0^t\W_{p}(P^*_s\mu,P^*_s\nu)^p \d s,\ \ t\in[0,T].
\end{align*}
So, by the Gronwall inequality, we complete the proof.
\end{proof}

\begin{rem}\label{Re(Est)}
Under the same conditions as Theorem \ref{WP}, we obtain that for any $t\in[0,T]$,
\beg{align*}
\E\bigg(\sup_{s\in[0,t]}|\varrho^\mu_s-\varrho^\nu_s|^p\bigg)\leq C_{p,T,\kappa,\ti H}t^{p\ti H}\W_p(\mu,\nu)^p.
\end{align*}
Here we have set $\varrho^\mu_s:=\int_0^s\ti\si_r(P_r^*\mu)\d\ti B^{\ti H}_r$ for all $s\in[0,T]$ and $\mu\in\sP(\R^d)$.
Indeed, combining \cite[(3.8) in the proof of Theorem 3.1]{FHSY} with \eqref{1Est} yields
\beg{align*}
&\E\bigg(\sup_{s\in[0,t]}|\varrho^\mu_s-\varrho^\nu_s|^p\bigg)
=\E\bigg(\sup_{s\in[0,t]}\left|\int_0^s(\ti\si_r(P_r^*\mu)-\ti\si_r(P_r^*\nu))\d\ti B^{\ti H}_r\right|^p\bigg)\cr
\leq&C_{p,\ti H}\kappa^p_tt^{p\ti H-1}\int_0^t\W_p(P^*_r\mu,P^*_r\nu)^p\d r
\leq C_{p,T,\kappa,\ti H}t^{p\ti H}\W_p(\mu,\nu)^p.
\end{align*}
\end{rem}

\section{Regularities of DDSDEs by fractional noises}

The main objective of this section concerns the regularities for \eqref{GeEq}.
More precisely, for any $t\in[0,T],\mu\in\sP_p(\R^d)$ and $f\in\sB_b(\R^d)$, let
\beg{align}\label{DeOp}
(P_t f)(\mu)=\int_{\R^d} f\d(P^*_t\mu)
\end{align}
with $P^*_t\mu:=\sL_{X_t^\mu}$ for $X_t^\mu$ solving \eqref{GeEq} with initial distribution $\mu$, and then introduce the functionals
\beg{align*}
\sP_p(\R^d)\ni\mu\mapsto(P_t f)(\mu),\ \ t\in[0,T], \ f\in\sB_b(\R^d).
\end{align*}
Based on the coupling argument and a careful analysis of fractional derivative operator, we shall establish the log-Harnack inequalities and the Bismut formulas for these functionals in
both non-degenerate and degenerate cases.

\subsection{The non-degenerate case}

This part is devoted to the regularities for the non-degenerate case of \eqref{GeEq}.
We begin with the following assumption.
\beg{enumerate}
\item[\textbf{(H1')}] For every $t\in[0,T]$, $b_t(\cdot,\cdot)\in C^{1,(1,0)}(\R^d\times\sP_p(\R^d))$.
Moreover, there exists a non-decreasing function $\kappa_\cdot$ such that for any $t\in[0,T],\ x,y\in\R^d,\ \mu,\nu\in\sP_p(\R^d)$,
\beg{align*}
\|\nabla b_t(\cdot,\mu)(x)\|+|D^Lb_t(x,\cdot)(\mu)(y)|\le \kappa_t, \ \ \|\ti\si_t(\mu)-\ti\si_t(\nu)\|\le \kappa_t\W_p(\mu,\nu),
\end{align*}
and $|b_t(0,\de_0)|+\|\si_t\|+\|\ti\si_t(\de_0)\|\leq \kappa_t$.
\end{enumerate}
Observe that with the help of the fundamental theorem for Bochner integral (see, for instance, \cite[Proposition A.2.3]{LWbook})
and the definitions of $L$-derivative and the Wasserstein distance, \textsc{\textbf{(H1')}} implies that for each $p\geq1$,
\beg{align*}
|b_t(x,\mu)-b_t(y,\nu)|\le\kappa_t(|x-y|+\W_p(\mu,\nu)),\ \  t\in[0,T],\ x,y\in\R^d,\ \mu,\nu\in\sP_p(\R^d).
\end{align*}
So, according to Theorem \ref{WP}, \eqref{GeEq} admits a unique solution.
To investigate the regularities, in additional to \textsc{\textbf{(H1')}}, we also need the following condition.
\beg{enumerate}
\item[\textsc{\textbf{(H2)}}] There exists a constant $\ti{\kappa}>0$ such that

\item[(i)]  for any $t,s\in[0,T],\ x,y,z_1,z_2\in\R^d,\ \mu,\nu\in\sP_p(\R^d)$,
\beg{align*}
&\|\nabla b_t(\cdot,\mu)(x)-\nabla b_s(\cdot,\nu)(y)\|+|D^Lb_t(x,\cdot)(\mu)(z_1)-D^Lb_s(y,\cdot)(\nu)(z_2)|\cr
&\le \ti{\kappa}(|t-s|^{\alpha}+|x-y|^{\beta}+|z_1-z_2|^{\gamma}+\W_p(\mu,\nu)),
\end{align*}
where $\alpha\in(H-1/2,1]$ and $\beta,\gamma\in(1-1/(2H),1]$.

\item[(ii)] $\si$ is invertible and $\si^{-1}$ is H\"{o}lder continuous of order $\delta\in(H-1/2,1]$:
\beg{align*}
\|\si^{-1}(t)-\si^{-1}(s)\|\le\ti{\kappa}|t-s|^{\delta}, \ \  t,s\in[0,T].
\end{align*}
\end{enumerate}

\subsubsection{Log-Harnack inequality}

Our main goal in the current part is to prove the following log-Harnack inequality.

\begin{thm}\label{Th(Har)}
Consider Eq. \eqref{GeEq}.
If one of the two following assumptions holds:
\beg{enumerate}
\item[(I)] $H\in(1/2,1)$, $b,\si,\ti\si$ satisfy \textsc{\textbf{(H1')}}, \textsc{\textbf{(H2)}} and $p\geq 2(1+\be)$;
\item[(II)] $H\in(0,1/2), b,\ti\si$ satisfies \textsc{\textbf{(H1)}}, $\si_t$ does not depend on $t$ and $p\geq 2$.
\end{enumerate}
Then for any $t\in(0,T], \mu,\nu\in\sP_p(\R^d)$ and $0<f\in\sB_b(\R^d)$,
\beg{align}\label{1Th(Har)}
(P_t\log f)(\nu)\leq\log(P_t f)(\mu)+\varpi(H),
\end{align}
where
\begin{equation}\nonumber
\varpi(H)=
\left\{
\begin{array}{ll}\vspace{0.3cm}
 C_{T,\kappa,\ti\kappa,H,\ti H}\left(1+\W_p(\mu,\nu)^{2\be}+\ff 1 {t^{2H}}\right)\W_p(\mu,\nu)^2,\ \ \ \ \ \ \ H\in(1/2,1),\\
 C_{T,\kappa,H,\ti H}\left(1+\ff 1{t^{2H}}\right)\W_p(\mu,\nu)^2,\ \ \ \ \ \ \ \ \ \ \ \ \ \ \ \ \ \ \ \ \ \ \ \ \ \ H\in(0,1/2).
\end{array} \right.
\end{equation}
\end{thm}

\begin{rem}\label{ReTh(Har)}
The log-Harnack inequality obtained above is equivalent to the following entropy-cost estimate
\beg{align*}
\mathrm{Ent}(P_t^*\nu|P_t^*\mu)\leq\varpi(H),\ \ t\in(0,T], \mu,\nu\in\sP_p(\R^d),
\end{align*}
where $\mathrm{Ent}(P_t^*\nu|P_t^*\mu)$ is the relative entropy of $P_t^*\nu$ with respect to $P_t^*\mu$ and $p$ is given as in Theorem \ref{Th(Har)}.
\end{rem}

\emph{Proof of Theorem \ref{Th(Har)}.}
For every $\mu,\nu\in\sP_p(\R^d)$, choose $\sF_0$-measurable $X_0^\mu$ and $X_0^\nu$ such that $\sL_{X_0^\mu}=\mu,\sL_{X_0^\nu}=\nu$ and
\beg{align}\label{1PfTh(Har)}
\E|X_0^\mu-X_0^\nu|^p=\W_p(\mu,\nu)^p.
\end{align}
Let $X_t^\mu$ and $X_t^\nu$ be two solutions to \eqref{GeEq} such that $\sL_{X_0^\mu}=\mu$ and $\sL_{X_0^\nu}=\nu$, respectively,
which yields that $\sL_{X_t^\mu}=P_t^*\mu$ and  $\sL_{X_t^\nu}=P_t^*\nu$.

For fixed $t_0\in(0,T]$, we first consider the following coupling DDSDE:
\beg{align}\label{2PfTh(Har)}
\d Y_t=&\left[b_t(X_t^\mu,P_t^*\mu)+\ff 1 {t_0}(X_0^\mu-X_0^\nu+\varrho^\mu_{t_0}-\varrho^\nu_{t_0})\right]\d t\cr
&+\si_t\d B^{H}_t+\ti\si_t(P_t^*\nu)\d\ti B^{\ti H}_t,\ \ t\in[0,t_0]
\end{align}
with $Y_0=X_0^\nu$.
Recall that  $\varrho^\mu_s=\int_0^s\ti\si_r(P_r^*\mu)\d\ti B^{\ti H}_r, (s,\mu)\in[0,T]\times\sP_p(\R^d)$ is defined in
Remark \ref{Re(Est)}.
Taking into account of this and \eqref{GeEq} for $(X_t^\mu,P_t^*\mu)$ replacing $(X_t,\sL_{X_t})$, we obtain
\beg{align}\label{3PfTh(Har)}
Y_t-X_t^\mu=\ff {t-t_0} {t_0}(X_0^\mu-X_0^\nu)+\ff {t} {t_0}(\varrho^\mu_{t_0}-\varrho^\nu_{t_0})+\varrho^\nu_t-\varrho^\mu_t,\ \ t\in[0,t_0].
\end{align}
In particular, one has $Y_{t_0}=X_{t_0}^\mu$.

Next, we intend to express $P_{t_0}f(\nu)$ in terms of $Y_{t_0}$.
To this end, we first rewrite Eq. \eqref{2PfTh(Har)} as
\beg{align}\label{4PfTh(Har)}
\d Y_t=b_t(Y_t,P_t^*\nu)\d t+\si_t\d\bar{B}^{H}_t+\ti\si_t(P_t^*\nu)\d\ti B^{\ti H}_t,\ \ t\in[0,t_0],
\end{align}
where
\beg{align*}
\bar{B}^{H}_t:=B^{H}_t-\int_0^t\si^{-1}_s\zeta_s\d s
=\int_0^tK_H(t,s)\left(\d W_s-K_H^{-1}\left(\int_0^\cdot\si^{-1}_r\zeta_r\d r\right)(s)\d s\right)
\end{align*}
with
\beg{align*}
\zeta_s:=b_s(Y_s,P_s^*\nu)-b_s(X_s^\mu,P_s^*\mu)-\ff 1 {t_0}(X_0^\mu-X_0^\nu+\varrho^\mu_{t_0}-\varrho^\nu_{t_0}).
\end{align*}
Set
\beg{align*}
R^{\ti H,0}:=\exp\left[\int_0^{t_0}\left\langle K_H^{-1}\left(\int_0^\cdot\si^{-1}_r\zeta_r\d r\right)(s),\d W_s\right\rangle-\ff 1 2\int_0^{t_0}\left|K_H^{-1}\left(\int_0^\cdot\si^{-1}_r\zeta_r\d r\right)(s)\right|^2\d s\right].
\end{align*}
On one hand, with the help of Remark \ref{Re(GirsT)} (i) below and the fractional Girsanov theorem (see, e.g., \cite[Theorem 4.9]{DU98} or \cite[Theorem 2]{NO02}), we know that $(\bar{B}^{H}_t)_{t\in[0,t_0]}$ is a $d$-dimensional fractional Brownian motion under the conditional probability $R^{\ti H,0}\d\P^{\ti H,0}$.
Here and in the sequel, we use $\P^{\ti H,0}$ and $\E^{\ti H,0}$ to denote the conditional probability and the conditional expectation given both $\ti B^{\ti H}$ and $\sF_0$, i.e.
\beg{align*}
\P^{\ti H,0}=\P(\ \cdot\ |\ti B^{\ti H},\sF_0),\ \ \E^{\ti H,0}=\E(\ \cdot\ |\ti B^{\ti H},\sF_0).
\end{align*}
On the other hand, let $\bar{Y}_t=Y_t-\varrho^\nu_t$ and then \eqref{4PfTh(Har)} can be written as
\beg{align*}
\d \bar{Y}_t=b_t(\bar{Y}_t+\varrho^\nu_t,P_t^*\nu)\d t+\si_t\d\bar{B}^{H}_t,\ \ t\in[0,t_0], \ \ \bar Y_0=Y_0=X_0^\nu.
\end{align*}
Note that $\bar{X}_\cdot^\nu:=X_\cdot^\nu-\varrho^\nu_\cdot$ satisfies SDE of the same form
\beg{align*}
\d \bar{X}^\nu_t=b_t(\bar{X}^\nu_t+\varrho^\nu_t,P_t^*\nu)\d t+\si_t\d B^{H}_t,\ \ t\in[0,t_0],\ \ \bar X_0^\nu=X_0^\nu.
\end{align*}
Therefore, by the weak uniqueness of the solution we derive that the law of $\bar{Y}_{t_0}$ under $R^{\ti H,0}\d\P^{\ti H,0}$ is the same as that of $ \bar{X}^\nu_{t_0}$ under $\P^{\ti H,0}$.
Consequently, we conclude that the law of $Y_{t_0}=\bar{Y}_{t_0}+\varrho^\nu_{t_0}$ under $R^{\ti H,0}\d\P^{\ti H,0}$
is also the same as one of $X_{t_0}^\nu=\bar{X}^\nu_{t_0}+\varrho^\nu_{t_0}$ under $\P^{\ti H,0}$
due to the fact that  $\varrho^\nu_{t_0}$ is deterministic given $\ti B^{\ti H}$.
This, along with $Y_{t_0}=X_{t_0}^\mu$, yields that for any $f\in\sB_b(\R^d)$,
\beg{align}\label{5PfTh(Har)}
(P_{t_0}^{\ti H,0}f)(X_0^\nu):=\E^{\ti H,0}f(X_{t_0}^\nu)=\E_{R^{\ti H,0}\P^{\ti H,0}}f(Y_{t_0})=\E_{R^{\ti H,0}\P^{\ti H,0}}f(X_{t_0}^\mu).
\end{align}

Now, owing to \eqref{DeOp} and \eqref{5PfTh(Har)}, we deduce that for every $0<f\in\sB_b(\R^d)$,
\beg{align}\label{6PfTh(Har)}
&(P_{t_0}\log f)(\nu)=\E\left[\E^{\ti H,0}(\log f(X_{t_0}^\nu))\right]=\E\left[(P_{t_0}^{\ti H,0}\log f)(X_0^\nu)\right]\cr
=&\E\left[\E_{R^{\ti H,0}\P^{\ti H,0}}\log f(X_{t_0}^\mu)\right]=\E\left[\E^{\ti H,0}\left(R^{\ti H,0}\log f(X_{t_0}^\mu)\right)\right]\cr
\leq&\E\left[\log\E^{\ti H,0}f(X_{t_0}^\mu)+\E^{\ti H,0}\left(R^{\ti H,0}\log R^{\ti H,0}\right)\right]\cr
=&\E\left[\log(P_{t_0}^{\ti H,0}f)(X_0^\mu)\right]+\ff 1 2\E\left[\E^{\ti H,0}\left(\int_0^{t_0}\left|K_H^{-1}\left(\int_0^\cdot\si^{-1}_r\zeta_r\d r\right)(s)\right|^2\d s\right)\right],
\end{align}
where we use the Young inequality (see, e.g., \cite[Lemma 2.4]{ATW09}) in the inequality.\\
Using the Jensen inequality and Lemma \ref{GirsT} below, we have
\beg{align}\label{7PfTh(Har)}
(P_{t_0}\log f)(\nu)\leq&\log\E\left((P_{t_0}^{\ti H,0}f)(X_0^\mu)\right)+\ff 1 2\E\vartheta(H)\cr
=&\log(P_{t_0}f)(\mu)+\ff 1 2\E\vartheta(H),\ \ t_0\in(0,T], \mu,\nu\in\sP_p(\R^d).
\end{align}
Consequently, using Remark \ref{Re(GirsT)} (ii), we obtain the desired relations.
Our proof is now finished.
\qed

The following lemma and Remark \ref{Re(GirsT)} below consist of estimates on the function $K_H^{-1}\left(\int_0^\cdot\si^{-1}_r\zeta_r\d r\right)(s)$,
which may contribute to the study of the Girsanov transformation for the fractional Brownian motion case and then the log-Harnack inequality \eqref{1Th(Har)}.
Before going on, for any given continuous function $f:[0,T]\ra\R^d$ and H\"{o}lder continuous function $g:[0,T]\ra\R^d$ of order $\alpha\in(0,1)$, we put
\beg{align*}
\|f\|_\infty:=\sup_{t\in[0,T]}|f(t)|,\ \  \|g\|_\alpha:=\sup_{0\leq s<t\leq T}\ff {|g(t)-g(s)|}{(t-s)^\alpha}.
\end{align*}

\begin{lem}\label{GirsT}
Let the assumptions in Theorem \ref{Th(Har)} hold, then for any $\mu,\nu\in\sP_p(\R^d)$ with $p\geq 2(1+\be)$ if $H\in(1/2,1)$ or $p\geq 2$ if $H\in(0,1/2)$,
\beg{align*}
\E^{\ti H,0}\left(\int_0^{t_0}\left|K_H^{-1}\left(\int_0^\cdot\si^{-1}_r\zeta_r\d r\right)(s)\right|^2\d s\right)\leq\vartheta(H),
\end{align*}
where
\begin{equation}\nonumber
\vartheta(H)=
\left\{
\begin{array}{ll}\vspace{0.3cm}
C_{T,\kappa,\ti\kappa,H,\ti H}\Bigg[\W_p(\mu,\nu)^2+\bigg(\ff 1 {t_0^{2H}}+\|\varrho^\mu_\cdot\|^{2\be}_{\ti H-\varsigma_1}+
\|\varrho^\nu_\cdot\|^{2\be}_{\ti H-\varsigma_2}+\psi^{2\beta}(X_0,\varrho)\bigg)\psi^2(X_0,\varrho)\cr
\qquad\qquad\quad+\left(1+|X_0^\mu|^{2\be}+\|\varrho^\mu_\cdot\|_\infty^{2\be}+\left\|\varrho^\mu_\cdot\right\|^{2\be}_{\ti H-\varsigma_1}\right)(\W_p(\mu,\nu)^2+\psi^2(X_0,\varrho))\cr
\qquad\qquad\quad+\int_0^{t_0}s^{2H-1}\left(\int_0^s\frac{|\varrho^\nu_s-\varrho^\mu_s-(\varrho^\nu_r-\varrho^\mu_r)|}{(s-r)^{\frac{1}{2}+H}}r^{\ff 1 2-H}\d r\right)^2\d s\Bigg],\ \ \ \ \ \ \ H\in(1/2,1),\\
 C_{T,\kappa,H,\ti H}\left(\ff {\psi^2(X_0,\varrho)} {t_0^{2H}}+\W_p(\mu,\nu)^2\right),\ \ \ \ \ \ \ \ \ \ \ \ \ \ \ \ \ \ \ \ \ \ \ \ \ \ \ \ \ \ \ \ \ \ \ \ \  H\in(0,1/2),
\end{array} \right.
\end{equation}
with $\psi(X_0,\varrho):=|X_0^\mu-X_0^\nu|+\sup_{s\in[0,t_0]}|\varrho^\mu_s-\varrho^\nu_s|$ and $\varsigma_i\in(0,1/2),i=1,2,3$.
\end{lem}

\begin{proof}

We start by dealing with the case $H\in(1/2,1)$
By \eqref{InOp} and \eqref{FrDe}, we get
\beg{align}\label{1PfGirsT}
&K_H^{-1}\left(\int_0^\cdot\si^{-1}_r\zeta_r\d r\right)(s)
=s^{H-\frac{1}{2}}D^{H-\frac{1}{2}}_{0+}\left[\cdot^{\frac{1}{2}-H}\si^{-1}_\cdot\zeta_\cdot\right](s)\cr
=&\frac{H-\ff 1 2}{\Gamma(\frac{3}{2}-H)}
\Bigg[\ff {s^{\frac{1}{2}-H}\si^{-1}_s\zeta_s}{H-\ff 1 2}
+s^{H-\frac{1}{2}}\si^{-1}_s\zeta_s\int_0^s\frac{s^{\frac{1}{2}-H}-r^{\frac{1}{2}-H}}{(s-r)^{\frac{1}{2}+H}}
\d r\cr
&\qquad\qquad\quad+s^{H-\frac{1}{2}}\zeta_s\int_0^s\frac{\si^{-1}_s-\si^{-1}_r}{(s-r)^{\frac{1}{2}+H}}r^{\ff 1 2-H}\d r\cr
&\qquad\qquad\quad+s^{H-\frac{1}{2}}\int_0^s\frac{\zeta_s-\zeta_r}{(s-r)^{\frac{1}{2}+H}}\si^{-1}_rr^{\ff 1 2-H}\d r\Bigg]\cr
=:&\frac{H-\ff 1 2}{\Gamma(\frac{3}{2}-H)}[I_1(s)+I_2(s)+I_3(s)+I_4(s)].
\end{align}
From \textsc{\textbf{(H1)}}, \eqref{3PfTh(Har)} and Theorem \ref{WP}, it follows that
\beg{align}\label{2PfGirsT}
|\zeta_s|\leq&\ff{\kappa_s(t_0-s)+1}{t_0}|X_0^\mu-X_0^\nu|+\ff{\kappa_s s+1}{t_0}|\varrho^\mu_{t_0}-\varrho^\nu_{t_0}|\cr
&+\kappa_s\left(|\varrho^\mu_s-\varrho^\nu_s|+C_{T,\ti H}\W_p(\mu,\nu)\right).
\end{align}
Besides, we have
\beg{align*}
\int_0^s\frac{r^{\frac{1}{2}-H}-s^{\frac{1}{2}-H}}{(s-r)^{\frac{1}{2}+H}}\d r
=s^{1-2H}\int_0^1\frac{r^{\frac{1}{2}-H}-1}{(1-r)^{\frac{1}{2}+H}}\d r<\infty.
\end{align*}
These, along with  \textbf{(H2)}(ii), lead to
\beg{align*}
\sum_{i=1}^3|I_i(s)|^2\leq& C_{T,\ti\kappa,H}(s^{1-2H}+s^{2\delta-2H+1})|\zeta_s|^2\cr
\leq&C_{T,\kappa,\ti\kappa,H,\ti H}(s^{1-2H}+s^{2\delta-2H+1})
\left(\ff {\psi^2(X_0,\varrho)} {t_0^2}+\W_p(\mu,\nu)^2\right),
\end{align*}
where we put $\psi(X_0,\varrho):=|X_0^\mu-X_0^\nu|+\sup_{s\in[0,t_0]}|\varrho^\mu_s-\varrho^\nu_s|$ for simplicity.\\
Then, we get
\beg{align}\label{3PfGirsT}
\sum_{i=1}^3\int_0^{t_0}|I_i(s)|^2\d s
\leq C_{T,\kappa,\ti\kappa,H,\ti H}\left(\ff {\psi^2(X_0,\varrho)} {t_0^{2H}}+\W_p(\mu,\nu)^2\right).
\end{align}
As for $I_4$, using \textbf{(H1')} and Lemma \ref{FoLD}, we deduce that for every $s\in[0,T]$,
\beg{align*}
&b_s(Y_s,P_s^*\nu)-b_s(X_s^\mu,P_s^*\mu)\cr
=&\int_0^1\ff d{\d\th}b_s(X_s^\mu+\th(Y_s-X_s^\mu),P_s^*\nu)\d\th+\int_0^1\ff d{\d\th}b_s(X_s^\mu,\sL_{X_s^\mu+\th(X_s^\nu-X_s^\mu)})\d\th\cr
=&\int_0^1\nabla b_s(\cdot,P_s^*\nu)(X_s^\mu+\th(Y_s-X_s^\mu))(Y_s-X_s^\mu)\d\th\cr
&+\int_0^1\left(\E\langle D^Lb_s(x,\cdot)(\sL_{X_s^{\mu,\nu}(\th)})(X_s^{\mu,\nu}(\th)),X_s^\nu-X_s^\mu\rangle\right)|_{x=X_s^\mu}\d\th,
\end{align*}
where for any $\th\in[0,1],X_s^{\mu,\nu}(\th):=X_s^\mu+\th(X_s^\nu-X_s^\mu)$.\\
Then by  \textbf{(H1')}, \textbf{(H2)}(i) and \eqref{3PfTh(Har)}, we have
\beg{align*}
&|\zeta_s-\zeta_r|=\left|b_s(Y_s,P_s^*\nu)-b_s(X_s^\mu,P_s^*\mu)-(b_r(Y_r,P_r^*\nu)-b_r(X_r^\mu,P_r^*\mu))\right|\cr
\leq&C_{T,\kappa,\ti\kappa,H,\ti H}\Bigg\{\Bigg[(s-r)^\alpha+\ff{s-r}{t_0}+\W_p(P_s^*\nu,P_r^*\nu)+\ff{(s-r)^\beta}{t_0^\beta}\psi^\beta(X_0,\varrho)
\cr
&\qquad\qquad\qquad+|\varrho^\nu_s-\varrho^\mu_s-(\varrho^\nu_r-\varrho^\mu_r)|^\beta+|X_s^\mu-X_r^\mu|^\beta\Bigg]\psi(X_0,\varrho)\cr
&\qquad\qquad\quad+\Big[(s-r)^\alpha+(\E|X_s^\mu-X_r^\mu|^p)^\ff \gamma p+(\E|X_s^\nu-X_r^\nu|^p)^\ff \gamma p\cr
&\qquad\qquad\qquad\quad+(\E|X_s^\mu-X_r^\mu|^p)^\ff 1 p+(\E|X_s^\nu-X_r^\nu|^p)^\ff 1 p+|X_s^\mu-X_r^\mu|^\beta\Big](\E|X_s^\mu-X_s^\nu|^p)^\ff 1 p\cr
&\qquad\quad\qquad+|\varrho^\nu_s-\varrho^\mu_s-(\varrho^\nu_r-\varrho^\mu_r)|+\E|(X_s^\mu-X_r^\mu)-(X_s^\nu-X_r^\nu)|\Bigg\}.
\end{align*}
Following respectively the same arguments as Theorem \ref{WP}, we derive that for any $s,r\in[0,T]$ and $\mu,\nu\in\sP_p(\R^d)$ with $p>\max\{1/H,1/\ti H\}$,
\beg{align*}
\E|X_s^\mu-X_r^\mu|^p\leq C_{T,\kappa,\ti\kappa,H,\ti H}|s-r|^{p(H\wedge\ti H)}
\end{align*}
and
\beg{align*}
\E|(X_s^\mu-X_r^\mu)-(X_s^\nu-X_r^\nu)|^p\leq  C_{T,\kappa,\ti\kappa,\ti H}(s-r)^{p\ti H}\W_p(\mu,\nu)^p.
\end{align*}
Consequently, combining these with \eqref{2PfEst} leads to
\beg{align}\label{4PfGirsT}
&|I_4(s)|^2\cr
\leq& C_{T,\kappa,\ti\kappa,H,\ti H}\Bigg[\left(s^{2(\alpha-H)+1}+s^{2(H\wedge\ti H-H)+1}+s^{2((H\wedge\ti H)\gamma-H)+1}\right)\W_p(\mu,\nu)^2\cr
&\qquad\qquad\quad+\left(s^{2(\alpha-H)+1}+s^{2(H\wedge\ti H-H)+1}+\ff{s^{3-2H}} {t_0^2}+\ff{s^{2(\beta-H)+1}} {t_0^{2\beta}}\psi^{2\beta}(X_0,\varrho)\right)\psi^2(X_0,\varrho)
\cr
&\qquad\qquad\quad+s^{2H-1}\left(\int_0^s\frac{|\varrho^\nu_s-\varrho^\mu_s-(\varrho^\nu_r-\varrho^\mu_r)|}{(s-r)^{\frac{1}{2}+H}}r^{\ff 1 2-H}\d r\right)^2\cr
&\qquad\qquad\quad+s^{2H-1}\left(\int_0^s\frac{|\varrho^\nu_s-\varrho^\mu_s-(\varrho^\nu_r-\varrho^\mu_r)|^\beta}{(s-r)^{\frac{1}{2}+H}}r^{\ff 1 2-H}\d r\right)^2\psi^2(X_0,\varrho)\cr
&\qquad\qquad\quad+s^{2H-1}\left(\int_0^s\frac{|X_s^\mu-X_r^\mu|^\beta}{(s-r)^{\frac{1}{2}+H}}r^{\ff 1 2-H}\d r\right)^2(\W_p(\mu,\nu)^2+\psi^2(X_0,\varrho))\Bigg]\cr
\leq& C_{T,\kappa,\ti\kappa,H,\ti H}\Bigg[\left(s^{2(\alpha-H)+1}+s^{2(H\wedge\ti H-H)+1}+s^{2((H\wedge\ti H)\gamma-H)+1}\right)\W_p(\mu,\nu)^2\cr
&\qquad\qquad\quad+\bigg(s^{2(\alpha-H)+1}+s^{2(H\wedge\ti H-H)+1}+\ff{s^{3-2H}} {t_0^2}+\|\varrho^\mu_\cdot\|^{2\be}_{\ti H-\varsigma_1}s^{1+2(\ti H-\varsigma_1)\be-2\ti H}\cr
&\qquad\qquad\qquad\quad+\|\varrho^\nu_\cdot\|^{2\be}_{\ti H-\varsigma_2}s^{1+2(\ti H-\varsigma_2)\be-2\ti H}
+\ff{s^{2(\beta-H)+1}} {t_0^{2\beta}}\psi^{2\beta}(X_0,\varrho)\bigg)\psi^2(X_0,\varrho)\nonumber\\
&\qquad\qquad\quad+s^{2H-1}\left(\int_0^s\frac{|\varrho^\nu_s-\varrho^\mu_s-(\varrho^\nu_r-\varrho^\mu_r)|}{(s-r)^{\frac{1}{2}+H}}r^{\ff 1 2-H}\d r\right)^2\nonumber\\
&\qquad\qquad\quad+s^{2H-1}\left(\int_0^s\frac{|X_s^\mu-X_r^\mu|^\beta}{(s-r)^{\frac{1}{2}+H}}r^{\ff 1 2-H}\d r\right)^2(\W_p(\mu,\nu)^2+\psi^2(X_0,\varrho))\Bigg],
\end{align}
where the last inequality is due to the H\"{o}lder continuity of $\varrho^\mu_\cdot$ and $\varrho^\nu_\cdot$ of order $\ti H-\varsigma_1$ and  $\ti H-\varsigma_2$ with $\varsigma_i\in(0,1/2),i=1,2$, respectively.\\
Observe that there hold
\beg{align*}
\sup_{t\in[0,T]}|X_t^\mu|\le& C_{T,\kappa,H,\ti H}\left(1+(\E|X_0^\mu|^p)^\ff 1 p+|X_0^\mu|+\left\|\int_0^\cdot\si_t\d B^H_t\right\|_\infty+\left\|\varrho^\mu_\cdot\right\|_\infty\right)\cr
=:&C_{T,\kappa,H,\ti H}\Upsilon_\mu
\end{align*}
and
\beg{align*}
&\E\bigg(\sup_{t\in[0,T]}|X_t^\mu|^p\bigg)\le C_{T,\kappa,H,\ti H}\left(1+\E|X_0^\mu|^p\right).
\end{align*}
Then, by \textsc{\textbf{(H1')}} we get
\beg{align*}
\left|\int_r^sb_t(X_t^{\mu},\sL_{X_t^\mu})\d t\right|
\le& \kappa(T)\left[1+\sup_{t\in[0,T]}|X_t^{\mu}|+\Big(\E\sup_{t\in[0,T]}|X_t^{\mu}|^2\Big)^\ff1 2\right](s-r)\cr
\le& C_{T,\kappa,H,\ti H}\Upsilon_\mu(s-r).
\end{align*}
As a consequence, we obtain
\beg{align*}
&s^{2H-1}\left(\int_0^s\frac{|X_s^\mu-X_r^\mu|^\beta}{(s-r)^{\frac{1}{2}+H}}r^{\ff 1 2-H}\d r\right)^2\cr
\le&3s^{2H-1}\bigg(\int_0^s\ff {|\int_r^sb_t(X_t^\mu,\sL_{X_t^\mu})\d t|^{\be}}{(s-r)^{\ff 1 2+H}}r^{\ff 1 2-H}\d r\bigg)^2
+3s^{2H-1}\bigg(\int_0^s\ff {|\int_r^s\si_t\d B^H_t|^{\be}}{(s-r)^{\ff 1 2+H}}r^{\ff 1 2-H}\d r\bigg)^2\nonumber\\
&+3s^{2H-1}\bigg(\int_0^s\ff {|\varrho^\mu_s-\varrho^\mu_r|^{\be}}{(s-r)^{\ff 1 2+H}}r^{\ff 1 2-H}\d r\bigg)^2\nonumber\\
\le& C_{T,\kappa,H,\ti H}\Upsilon_\mu^{2\be}s^{1+2(\be-H)}
+C_H\left\|\int_0^\cdot\si_t\d B^H_t\right\|^{2\be}_{H-\varsigma_3}s^{1+2(H-\varsigma_3)\be-2H}\cr
&+C_{\ti H}\left\|\varrho^\mu_\cdot\right\|^{2\be}_{\ti H-\varsigma_1}s^{1+2(\ti H-\varsigma_1)\be-2\ti H}.
\end{align*}
Here we have used the H\"{o}lder continuity of $\int_0^\cdot\si_t\d B^H_t$ of order $H-\varsigma_3$ with $\varsigma_3\in(0,1/2)$.\\
Substituting this into \eqref{4PfGirsT} and integrating on the interval $[0,t_0]$ yields
\beg{align}\label{5PfGirsT}
&\int_0^{t_0}|I_4(s)|^2\d s\cr
\leq& C_{T,\kappa,\ti\kappa,H,\ti H}\Bigg[\left(t_0^{2(\alpha-H+1)}+t_0^{2(H\wedge\ti H-H+1)}+t_0^{2((H\wedge\ti H)\gamma-H+1)}\right)\W_p(\mu,\nu)^2\cr
&\qquad\qquad\quad+\bigg(t_0^{2(\alpha-H+1)}+t_0^{2(H\wedge\ti H-H+1)}+t_0^{2(1-H)}+\|\varrho^\mu_\cdot\|^{2\be}_{\ti H-\varsigma_1}t_0^{2(1+(\ti H-\varsigma_1)\be-\ti H)}\cr
&\qquad\qquad\qquad\quad+\|\varrho^\nu_\cdot\|^{2\be}_{\ti H-\varsigma_2}t_0^{2(1+(\ti H-\varsigma_2)\be-\ti H)}+t_0^{2(1-H)}\psi^{2\beta}(X_0,\varrho)\bigg)\psi^2(X_0,\varrho)\cr
&\qquad\qquad\quad+\int_0^{t_0}s^{2H-1}\left(\int_0^s\frac{|\varrho^\nu_s-\varrho^\mu_s-(\varrho^\nu_r-\varrho^\mu_r)|}{(s-r)^{\frac{1}{2}+H}}r^{\ff 1 2-H}\d r\right)^2\d s\cr
&\qquad\qquad\quad+\Bigg(\Upsilon_\mu^{2\be}t_0^{2(1+\be-H)}
+\left\|\int_0^\cdot\si_t\d B^H_t\right\|^{2\be}_{H-\varsigma_3}t_0^{2(1+(H-\varsigma_3)\be-H)}\cr
&\qquad\qquad\qquad\quad+\left\|\varrho^\mu_\cdot\right\|^{2\be}_{\ti H-\varsigma_1}t_0^{2(1+(\ti H-\varsigma_1)\be-\ti H)}\Bigg)(\W_p(\mu,\nu)^2+\psi^2(X_0,\varrho))\Bigg].
\end{align}
This, together with \eqref{3PfGirsT} and \eqref{1PfGirsT}, implies
\beg{align*}
&\E^{\ti H,0}\int_0^{t_0}\left|K_H^{-1}\left(\int_0^\cdot\si^{-1}_r\zeta_r\d r\right)(s)\right|^2\d s\cr
\leq& C_{T,\kappa,\ti\kappa,H,\ti H}\Bigg[
\left(1+t_0^{2(\alpha-H+1)}+t_0^{2(H\wedge\ti H-H+1)}+t_0^{2((H\wedge\ti H)\gamma-H+1)}\right)\W_p(\mu,\nu)^2\cr
&\qquad\qquad\quad+\bigg(\ff 1 {t_0^{2H}}+t_0^{2(\alpha-H+1)}+t_0^{2(H\wedge\ti H-H+1)}+t_0^{2(1-H)}+\|\varrho^\mu_\cdot\|^{2\be}_{\ti H-\varsigma_1}t_0^{2(1+(\ti H-\varsigma_1)\be-\ti H)}\cr
&\qquad\qquad\qquad\quad+\|\varrho^\nu_\cdot\|^{2\be}_{\ti H-\varsigma_2}t_0^{2(1+(\ti H-\varsigma_2)\be-\ti H)}+t_0^{2(1-H)}\psi^{2\beta}(X_0,\varrho)\bigg)\psi^2(X_0,\varrho)\cr
&\qquad\qquad\quad+\int_0^{t_0}s^{2H-1}\left(\int_0^s\frac{|\varrho^\nu_s-\varrho^\mu_s-(\varrho^\nu_r-\varrho^\mu_r)|}{(s-r)^{\frac{1}{2}+H}}r^{\ff 1 2-H}\d r\right)^2\d s\cr
&\qquad\qquad\quad+\Bigg(\left(1+|X_0^\mu|^{2\be}+\|\varrho^\mu_\cdot\|_\infty^{2\be}\right)t_0^{2(1+\be-H)}
+t_0^{2(1+(H-\varsigma_3)\be-H)}\cr
&\qquad\qquad\qquad\quad+\left\|\varrho^\mu_\cdot\right\|^{2\be}_{\ti H-\varsigma_1}t_0^{2(1+(\ti H-\varsigma_1)\be-\ti H)}\Bigg)(\W_p(\mu,\nu)^2+\psi^2(X_0,\varrho))\Bigg]\cr
\leq&C_{T,\kappa,\ti\kappa,H,\ti H}\Bigg[\W_p(\mu,\nu)^2+\bigg(\ff 1 {t_0^{2H}}+\|\varrho^\mu_\cdot\|^{2\be}_{\ti H-\varsigma_1}+
\|\varrho^\nu_\cdot\|^{2\be}_{\ti H-\varsigma_2}+\psi^{2\beta}(X_0,\varrho)\bigg)\psi^2(X_0,\varrho)\cr
&\qquad\qquad\quad+\left(1+|X_0^\mu|^{2\be}+\|\varrho^\mu_\cdot\|_\infty^{2\be}+\left\|\varrho^\mu_\cdot\right\|^{2\be}_{\ti H-\varsigma_1}\right)(\W_p(\mu,\nu)^2+\psi^2(X_0,\varrho))\cr
&\qquad\qquad\quad+\int_0^{t_0}s^{2H-1}\left(\int_0^s\frac{|\varrho^\nu_s-\varrho^\mu_s-(\varrho^\nu_r-\varrho^\mu_r)|}{(s-r)^{\frac{1}{2}+H}}r^{\ff 1 2-H}\d r\right)^2\d s\Bigg]\cr
\end{align*}
Then we get the desired claim.

We now move on to the case $H\in(0,1/2)$.
According to \eqref{-HOF} and \eqref{FrIn}, we get
\beg{align}\label{6PfGirsT}
&\left|K_H^{-1}\left(\int_0^\cdot\si^{-1}_r\zeta_r\d r\right)(s)\right|
=\left|s^{H-\frac{1}{2}}I^{\frac{1}{2}-H}_{0+}
\left[\cdot^{\frac{1}{2}-H}\si^{-1}\zeta_\cdot\right](s)\right|\cr
=&\left|\frac{\si^{-1}s^{H-\ff 1 2}}{\Gamma(\frac{1}{2}-H)}\int_0^s\frac{r^{\ff 1 2-H}\zeta_r}{(s-r)^{\frac{1}{2}+H}}\d r\right|\cr
\leq&C_{T,\kappa,H,\ti H}s^{\ff 1 2-H}\left(\ff {\psi(X_0,\varrho)} {t_0}+\W_p(\mu,\nu)\right),
\end{align}
where the last inequality is due to \eqref{2PfGirsT}.\\
Then, we obtain
\beg{align*}
\E^{\ti H,0}\int_0^{t_0}\left|K_H^{-1}\left(\int_0^\cdot\si^{-1}_r\zeta_r\d r\right)(s)\right|^2\d s\leq C_{T,\kappa,H,\ti H}\left(\ff {\psi^2(X_0,\varrho)} {t_0^{2H}}+\W_p(\mu,\nu)^2\right),
\end{align*}
which is the desired relation.
Our proof is now complete.
\end{proof}

\begin{rem}\label{Re(GirsT)}
(i) With the help of the Fernique theorem (see, e.g., \cite[Theorem 1.3.2]{Fe75} or \cite[Lemma 8]{Sa12}), by \eqref{3PfGirsT}, \eqref{5PfGirsT} and \eqref{6PfGirsT} we can conclude that
\beg{align*}
\E^{\ti H,0}\left(\exp\left\{\ff 1 2\int_0^{t_0}\left|K_H^{-1}\left(\int_0^\cdot\si^{-1}_r\zeta_r\d r\right)(s)\right|^2\d s\right\}\right)<\infty.
\end{align*}
(ii) Under the assumptions in Lemma \ref{GirsT}, we have
\begin{equation}\nonumber
\E\vartheta(H)\leq
\left\{
\begin{array}{ll}\vspace{0.3cm}
 C_{T,\kappa,\ti\kappa,H,\ti H}\left(1+\W_p(\mu,\nu)^{2\be}+\ff 1 {t_0^{2H}}\right)\W_p(\mu,\nu)^2,\ \ \ \ \ \ \ H\in(1/2,1),\\
 C_{T,\kappa,H,\ti H}\left(1+\ff 1{t_0^{2H}}\right)\W_p(\mu,\nu)^2,\ \ \ \ \ \ \ \ \ \ \ \ \ \ \ \ \ \ \ \ \ \ \ \ \ \ H\in(0,1/2).
\end{array} \right.
\end{equation}
Indeed, first observe that by Remark \ref{Re(Est)} and \eqref{1PfTh(Har)}, we derive that for any $1\leq q\leq p$,
\beg{align}\label{1Re(GirsT)}
\E\psi^q(X_0,\varrho)\leq C_{T,\kappa,\ti H}\left(1+{t_0}^{q\ti H}\right)\W_p(\mu,\nu)^q\leq C_{T,\kappa,\ti H}\W_p(\mu,\nu)^q.
\end{align}
If $H\in(0,1/2)$, then it is easy to see that for any $p\geq 2$,
\beg{align*}
\E\vartheta(H)\leq C_{T,\kappa,H,\ti H}\left(1+\ff 1{t_0^{2H}}\right)\W_p(\mu,\nu)^2.
\end{align*}
If $H\in(1/2,1)$, using the same lines as in Remark \eqref{Re(Est)} in the second inequality leads to
\beg{align}\label{2Re(GirsT)}
&\E\int_0^{t_0}s^{2H-1}\left(\int_0^s\frac{|\varrho^\nu_s-\varrho^\mu_s-(\varrho^\nu_r-\varrho^\mu_r)|}{(s-r)^{\frac{1}{2}+H}}r^{\ff 1 2-H}\d r\right)^2\d s\cr
\leq&\int_0^{t_0}s^{2H-1}\left(\int_0^s\ff{r^{1-2H}}{(s-r)^{1+2H-2\lambda_0}}\d r\right)\cdot\left(\int_0^s\ff{\E|\varrho^\nu_s-\varrho^\mu_s-(\varrho^\nu_r-\varrho^\mu_r)|^2}{(s-r)^{2\lambda_0}}\d r\right)\d s\cr
\leq&C_{\lambda_0,T,\kappa,H,\ti H}\int_0^{t_0}s^{2(\lambda_0-H)}\left(\int_0^s(s-r)^{2(\ti H-\lambda_0)}\d r\right)\d s\cdot\W_p(\mu,\nu)^2\cr
\leq&C_{\lambda_0,T,\kappa,H,\ti H}\W_p(\mu,\nu)^2,
\end{align}
where we take $\lambda_0$ such that $H<\lambda_0<\ti H+1/2$ and remark that $C_{\lambda_0,T,\kappa,H,\ti H}$ above may depend only on $T,\kappa,H,\ti H$ by choosing proper $\lambda_0$.\\
Then, by \eqref{1Re(GirsT)} and \eqref{2Re(GirsT)} one can verify that for any $p\geq 2(1+\be)$,
\beg{align*}
\E\vartheta(H)\leq C_{T,\kappa,\ti\kappa,H,\ti H}\left(1+\W_p(\mu,\nu)^{2\be}+\ff 1 {t_0^{2H}}\right)\W_p(\mu,\nu)^2.
\end{align*}
\end{rem}

\subsubsection{Bismut formula}

In this part, we focus on establishing a Bismut formula for the $L$-derivative of \eqref{GeEq}.
That is, for every $t\in(0,T],\mu\in\sP_p(\R^d)$ and $\phi\in L^p(\R^d\ra\R^d,\mu)$, we are to find an integrable random variable $M_t(\mu,\phi)$ such that
\beg{align*}
D^L_\phi(P_t f)(\mu)=\E\left(f(X_t^\mu)M_t(\mu,\phi)\right),\ \ f\in\sB_b(\R^d).
\end{align*}
Recall that for any $\mu\in\sP_p(\R^d)$, let $(X_t^\mu)_{t\in[0,T]}$ is the solution to \eqref{GeEq} with $\sL_{X_0^\mu}=\mu$ and $P^*_t\mu=\sL_{X_t^\mu}$ for every $t\in[0,T]$. For any $\ve\in[0,1]$ and  $\phi\in L^p(\R^d\ra\R^d,\mu)$, let $X_t^{\mu_{\ve,\phi}}$ denote the solution of \eqref{GeEq} with $X_0^{\mu_{\varepsilon,\phi}}=(\mathrm{Id}+\varepsilon\phi)(X_0^\mu)$.
In order to ease notations, we simply write $\mu_{\varepsilon,\phi}=\sL_{(\mathrm{Id}+\varepsilon\phi)(X_0^\mu)}$.

Next, we first consider the spatial derivative of $X_t^\mu$ along $\phi$:
\beg{align*}
 \na_\phi X_t^\mu:=\lim\limits_{\ve\ra0}\ff {X_t^{\mu_{\ve,\phi}}-X_t^\mu}\ve,\ \ t\in[0,T],\ \phi\in L^p(\R^d\ra\R^d,\mu).
\end{align*}
To this end, we impose the following assumption.
\beg{enumerate}
\item[\textbf{(H3)}] There exists a non-decreasing function $\kappa_\cdot$ such that
\beg{align*}
|D^L\ti\si_t(\mu)(x)|\le \kappa_t,\ \ t\in[0,T],\ x\in\R^d, \mu\in\sP_p(\R^d).
\end{align*}
\end{enumerate}

\begin{lem}\label{PD-GT}
Assume that \textsc{\textbf{(H1')}}, \textsc{\textbf{(H3)}} hold and $\si_t$ does not depend on $t$ if $H\in(0,1/2)$.
For any $\mu\in\sP_p(\R^d)$ and $\phi\in L^p(\R^d\ra\R^d,\mu)$ with
$p>\max\{1/H,1/\ti H\}$ if $H\in(1/2,1)$ or $p>1/\ti H$ if $H\in(0,1/2) $,
then the following assertions hold. \\
(i) $\na_\phi X_\cdot^\mu$ exists in $L^p(\Omega\ra C([0,T];\R^d),\P)$
such that $\na_\phi X_\cdot^\mu$ is the unique solution of the following linear SDE
\beg{align}\label{ParV-0}
\d G^\phi_t=&\left[\na_{G^\phi_t}b_t(\cdot,\sL_{X_t^\mu})(X_t^\mu)+\left(\E\<D^Lb_t(y,\cdot)(\sL_{X_t^\mu})(X_t^\mu), G^\phi_t\>\right)|_{y=X_t^\mu}\right]\d t\cr
&+\E\<D^L\ti\si_t(\sL_{X_t^\mu})(X_t^\mu), G^\phi_t\>\d\ti B^{\ti H}_t,\ \ G^\phi_0=\phi(X_0^\mu),
\end{align}
and
\beg{align*}
\E\bigg(\sup_{t\in[0,T]}|\na_\phi X_t^\mu|^p\bigg)\leq  C_{p,T,\kappa,H,\ti H}\|\phi\|^p_{L^p(\mu)}.
\end{align*}
(ii) It holds
\beg{align*}
\lim\limits_{\ve\downarrow0}\E\bigg(\sup_{s\in[0,t]}\left|\ff{\varrho_s^{\mu_{\ve,\phi}}-\varrho^\mu_s}\ve-\Lambda_s\right|^p\bigg)=0,
\end{align*}
where $\Lambda_\cdot$ is defined as
\beg{align*}
\Lambda_s:=\int_0^s\left\langle\E[\langle D^L\ti\si_r(P_r^*\mu)(X_r^\mu),\nabla_\phi X_r^\mu\rangle],\d\ti B^{\ti H}_r\right\rangle,\ \ s\in[0,T].
\end{align*}
\end{lem}

\begin{proof}
(i) We first set
\beg{align*}
\Pi_t^\ve:=\ff {X_t^{\mu_{\ve,\phi}}-X_t^\mu}\ve,\ \ t\in[0,T],\ve>0.
\end{align*}
By Lemma \ref{FoLD}, we deduce that for any $t\in[0,T]$,
\beg{align}\label{5PfPD-GT}
\d\Pi^\ve_t=&\ff {b_t(X_t^{\mu_{\ve,\phi}},\sL_{X_t^{\mu_{\ve,\phi}}})-b_t(X_t^\mu,\sL_{X_t^\mu})} \ve\d t
+\ff{\ti\si_t(\L_{X_t^{\mu_{\ve,\phi}}})-\ti\si_t(\L_{X_t^\mu})}\ve\d\ti B^{\ti H}_t \nonumber\\
=&\bigg[\int_0^1\Big(\na_{\Pi^\ve_t} b_t(\cdot,\sL_{X_t^{\mu_{\ve,\phi}}})(X^\ve_t(\th)) \nonumber\\
&\ \ \ \ \ \ \ \ \ \  +(\E\<D^Lb_t(x,\cdot)(\sL_{X^\ve_t(\th)})(X^\ve_t(\th)),\Pi^\ve_t\>)|_{x=X_t^\mu}\Big)\d\th\bigg]\d t\nonumber\\
&+\left[\int_0^1\E\<D^L\ti\si_t(\sL_{X^\ve_t(\th)})(X^\ve_t(\th)),\Pi^\ve_t\>\d\th\right]\d\ti B^{\ti H}_t,\ \ \Pi^\ve_0=\phi(X_0^\mu),
\end{align}
where $X^\ve_t(\th):=X_t^\mu+\th(X_t^{\mu_{\ve,\phi}}-X_t^\mu), \th\in[0,1]$.\\
On the other hand, it is easy to see that under \textbf{(H1')}, \eqref{ParV-0} has a unique solution.
Combining \eqref{ParV-0} with \eqref{5PfPD-GT} implies that for any $t\in[0,T]$,
\beg{align*}
\d(\Pi^\ve_t- G^\phi_t)&
=\left(\na_{\Pi^\ve_t-G^\phi_t}b_t(\cdot,\sL_{X_t^\mu})(X_t^\mu)+\Psi^\ve_1(t)\right)\d t\cr
&\quad+\left[\left(\E\<D^Lb_t(x,\cdot)(\sL_{X_t^\mu})(X_t^\mu),\Pi^\ve_t-G^\phi_t\>\right)\big|_{x=X_t^\mu}+\Psi^\ve_2(t)\right]\d t\cr
&\quad+\left(\E\<D^L\ti\si_t(\sL_{X_t^\mu})(X_t^\mu),\Pi^\ve_t-G^\phi_t\>+\Psi^\ve_3(t)\right)\d\ti B^{\ti H}_t,
\ \
\Pi^\ve_0- G^\phi_0=0,
\end{align*}
where
\beg{align*}
\Psi^\ve_1(t)&:=\int_0^1\left(\na_{\Pi^\ve_t} b_t(\cdot,\sL_{X_t^{\mu_{\ve,\phi}}})(X^\ve_t(\th))-\na_{\Pi^\ve_t}b_t(\cdot,\sL_{X_t^\mu})(X_t^\mu)\right)\d\th,\cr
\Psi^\ve_2(t)&:=\int_0^1\left(\E\<D^Lb_t(x,\cdot)(\sL_{X^\ve_t(\th)})(X^\ve_t(\th))-D^Lb_t(x,\cdot)(\sL_{X_t^\mu})(X_t^\mu),\Pi^\ve_t\>\right)|_{x=X_t^\mu}\d\th,\cr
\Psi^\ve_3(t)&:=\int_0^1\E\<D^L\ti\si_t(\sL_{X^\ve_t(\th)})(X^\ve_t(\th))-D^L\ti\si_t(\sL_{X_t^\mu})(X_t^\mu),\Pi^\ve_t\>\d\th.
\end{align*}
Then, using \textsc{\textbf{(H1')}} we have
\beg{align}\label{1PfPD-GT}
&|\Pi^\ve_t- G^\phi_t|^p\cr
\leq& C_{p,T,\kappa}\bigg[\int_0^t(|\Psi^\ve_1(s)|^p+|\Psi^\ve_2(s)|^p)\d s+\int_0^t\left(|\Pi^\ve_s- G^\phi_s|^p+\E|\Pi^\ve_s- G^\phi_s|^p\right)\d s\cr
&\qquad\quad+\left|\int_0^t\left(\E\<D^L\ti\si_s(\sL_{X_s^\mu})(X_s^\mu),\Pi^\ve_s-G^\phi_s\>+\Psi^\ve_3(s)\right)\d\ti B^{\ti H}_s\right|^p\Bigg].
\end{align}
By \cite[(3.5) in the proof of Theorem 3.1]{FHSY} and \textsc{\textbf{(H3)}}, we get
\beg{align}\label{2PfPD-GT}
&\E\left(\sup_{t\in[0,T]}\left|\int_0^t\left(\E\<D^L\ti\si_s(\sL_{X_s^\mu})(X_s^\mu),\Pi^\ve_s-G^\phi_s\>+\Psi^\ve_3(s)\right)\d\ti B^{\ti H}_s\right|^p\right)\cr
\leq&C_{p,T,\ti H}\int_0^T\left(\|\E\<D^L\ti\si_s(\sL_{X_s^\mu})(X_s^\mu),\Pi^\ve_s-G^\phi_s\>\|^p+\|\Psi^\ve_3(s)\|^p\right)\d s\cr
\leq&C_{p,T,\kappa,\ti H}\int_0^T\left(\E|\Pi^\ve_s-G^\phi_s|^p+\|\Psi^\ve_3(s)\|^p\right)\d s.
\end{align}
Additional, similar to \cite[Lemma 4.1 and (4.9)]{FHSY}, one has
\beg{align*}
\sup_{\ve\in(0,1]}\E\bigg(\sup_{t\in[0,T]}|\Pi^\ve_t|^p\bigg)+\E\bigg(\sup_{t\in[0,T]}|G^\phi_t|^p\bigg)
\leq C_{p,T,\kappa,H,\ti H}\|\phi\|^p_{L^p(\mu)}.
\end{align*}
Consequently, combining this with \eqref{1PfPD-GT}-\eqref{2PfPD-GT} and applying the Gronwall lemma, we obtain
\beg{align*}
\E\bigg(\sup_{t\in[0,T]}|\Pi^\ve_t-G^\phi_t|^p\bigg)\le C_{p,T,\kappa,\ti H}
\int_0^T\E\left(|\Psi^\ve_1(s)|^p+|\Psi^\ve_2(s)|^p+\|\Psi^\ve_3(s)\|^p\right)\d s.
\end{align*}
Then, following the argument to derive the assertion of \cite[Proposition 4.2]{FHSY} from \cite[(4.10) in the proof of Proposition 4.2]{FHSY},
we conclude that
\beg{align*}
\lim_{\ve\downarrow0}\E\bigg(\sup_{t\in[0,T]}|\Pi^\ve_t-G^\phi_t|^p\bigg)=0,
\end{align*}
which is exactly the first claim.

(ii) From \cite[(3.5) in the proof of Theorem 3.1]{{FHSY}} again, it follows that
\beg{align*}
&\E\bigg(\sup_{s\in[0,t]}\left|\ff{\varrho_s^{\mu_{\ve,\phi}}-\varrho^\mu_s}\ve-\Lambda_s\right|^p\bigg)\cr
\leq&C_{p,\ti H}t^{p\ti H-1}\int_0^t\left|\ff{\ti\si_r(P_r^*\mu_{\ve,\phi})-\ti\si_r(P_r^*\mu)}\ve-\E[\langle D^L\ti\si_r(P_r^*\mu)(X_r^\mu),\nabla_\phi X_r^\mu\rangle]\right|^p\d r.
\end{align*}
Observe that by \textbf{(H1')} and Theorem \ref{WP}, we get
\beg{align*}
|\ti\si_r(P_r^*\mu_{\ve,\phi})-\ti\si_r(P_r^*\mu)|\leq\kappa_r\W_\theta(P_r^*\mu_{\ve,\phi},P_r^*\mu)
\leq C_{p,T,\kappa,\ti H}\W_p(\mu_{\ve,\phi},\mu)\leq C_{p,T,\kappa,\ti H}\ve\|\phi\|_{L^p(\mu)}.
\end{align*}
Then, using \textbf{(H3)} and the assertion (i), and applying the dominated convergence theorem and Lemma \ref{FoLD}, we derive the second claim.
\end{proof}

Our main result in this part is the following.

\begin{thm}\label{Th(NBD)}
Consider Eq. \eqref{GeEq}.
If one of the two following assumptions holds:
\beg{enumerate}
\item[(I)] $H\in(1/2,1)$, $b,\si,\ti\si$ satisfy \textsc{\textbf{(H1')}}, \textsc{\textbf{(H2)}} and \textsc{\textbf{(H3)}};
\item[(II)] $H\in(0,1/2), b,\ti\si$ satisfies \textsc{\textbf{(H1')}}, \textsc{\textbf{(H3)}} and $\si_t$ does not depend on $t$,
\end{enumerate}
then for any $t\in(0,T],f\in\sB_b(\R^d),\mu\in\sP_p(\R^d)$ and $\phi\in L^p(\R^d\ra\R^d,\mu)$ with $p\geq 2(1+\be)$ if $H\in(1/2,1)$ or $p\geq 2$ if $H\in(0,1/2)$, $D^L_\phi(P_t f)(\mu)$ exists and satisfies
\beg{align}\label{1Th(NBD)}
D^L_\phi(P_{t}f)(\mu)
=\E\left(f(X_{t}^\mu)\int_0^{t}\left\langle K_H^{-1}\left(\int_0^\cdot\si^{-1}_r\Upsilon_{r,t}\d r\right)(s),\d W_s\right\rangle\right),
\end{align}
where $\Upsilon_{\cdot,\cdot}$ is given by
\beg{align*}
\Upsilon_{r,t}=&\ff{\phi(X_0^\mu)+\Lambda_{t}}{t}+\na b_r(\cdot,P_r^*\mu)(X_r^\mu)\left(\ff{t-r}{t}\phi(X_0^\mu)-\ff r {t}\Lambda_{t}+\Lambda_r\right)\cr
&+\E[\langle D^Lb_r(x,\cdot)(P_r^*\mu)(X_r^\mu),\nabla_\phi X_r^\mu\rangle]|_{x=X_r^\mu},\ \ 0\leq r<t\leq T
\end{align*}
with $\Lambda_\cdot$ defined in Lemma \ref{PD-GT}.
\end{thm}

\begin{proof}
Let $t_0\in(0,T]$ be fixed.
For $\ve\in(0,1]$, let $Y^\ve$ solve \eqref{2PfTh(Har)} with $\nu=\mu_{\varepsilon,\phi}$ and $Y_0=Y_0^\ve=(\mathrm{Id}+\varepsilon\phi)(X_0^\mu)$.
Correspondingly, \eqref{3PfTh(Har)} turns into
\beg{align}\label{1PfTh(NBD)}
Y_t^\ve-X_t^\mu=-\ve\ff {t-t_0} {t_0}\phi(X_0^\mu)+\ff {t} {t_0}(\varrho^\mu_{t_0}-\varrho^{\mu_{\varepsilon,\phi}}_{t_0})+\varrho^{\mu_{\varepsilon,\phi}}_t-\varrho^\mu_t,\ \ t\in[0,t_0],
\end{align}
which implies that $Y_{t_0}^\ve=X_{t_0}^\mu$.
Put
\beg{align*}
R_\ve^{\ti H,0}:=\exp\left[\int_0^{t_0}\left\langle K_H^{-1}\left(\int_0^\cdot\si^{-1}_r\zeta^\ve_r\d r\right)(s),\d W_s\right\rangle-\ff 1 2\int_0^{t_0}\left|K_H^{-1}\left(\int_0^\cdot\si^{-1}_r\zeta^\ve_r\d r\right)(s)\right|^2\d s\right]
\end{align*}
with
\beg{align*}
\zeta^\ve_s:=b_s(Y_s^\ve,P_s^*{\mu_{\varepsilon,\phi}})-b_s(X_s^\mu,P_s^*\mu)+\ff 1 {t_0}(\ve\phi(X_0^\mu)+\varrho^{\mu_{\varepsilon,\phi}}_{t_0}-\varrho^\mu_{t_0}).
\end{align*}
Similar to \eqref{5PfTh(Har)}, one has
\beg{align*}
(P_{t_0}^{\ti H,0}f)(X_0^{\mu_{\varepsilon,\phi}})=\E^{\ti H,0}\left(R_\ve^{\ti H,0}f(X_{t_0}^\mu)\right).
\end{align*}
Then, we arrive at
\beg{align}\label{3PfTh(NBD)}
\lim\limits_{\ve\downarrow0}\ff{(P_{t_0}^{\ti H,0}f)(X_0^{\mu_{\ve,\phi}})-(P_{t_0}^{\ti H,0}f)(X_0^\mu)}{\ve}=\lim\limits_{\ve\downarrow0}\E^{\ti H,0}\left(f(X_{t_0}^\mu)\ff {R_\ve^{\ti H,0}-1}\ve\right).
\end{align}
Note that we have
\beg{align}\label{4PfTh(NBD)}
&\lim\limits_{\ve\downarrow0}\E^{\ti H,0}\ff {R_\ve^{\ti H,0}-1}\ve\cr
=&\lim\limits_{\ve\downarrow0}\E^{\ti H,0}\ff{\int_0^{t_0}\left\langle K_H^{-1}\left(\int_0^\cdot\si^{-1}_r\zeta^\ve_r\d r\right)(s),\d W_s\right\rangle-\ff 1 2\int_0^{t_0}\left|K_H^{-1}\left(\int_0^\cdot\si^{-1}_r\zeta^\ve_r\d r\right)(s)\right|^2\d s}\ve\cr
=&\lim\limits_{\ve\downarrow0}\E^{\ti H,0}\ff{\int_0^{t_0}\left\langle K_H^{-1}\left(\int_0^\cdot\si^{-1}_r\zeta^\ve_r\d r\right)(s),\d W_s\right\rangle}\ve,
\end{align}
where the last equality is due to Remark \ref{Re(GirsT)} (ii) for $\mu_{\ve,\phi}$ replacing $\nu$ and the fact that $\W_p(\mu,\mu_{\ve,\phi})\leq \ve\|\phi\|_{L^p(\mu)}.$

Next, we handle the case $H\in(1/2,1)$ and $H\in(0,1/2)$ slightly.

\textsl{The case $H\in(1/2,1)$.}
In view of \eqref{InOp} and \eqref{FrDe}, one has
\beg{align}\label{2PfTh(NBD)}
&\int_0^{t_0}\left\langle K_H^{-1}\left(\int_0^\cdot\si^{-1}_r\zeta^\ve_r\d r\right)(s),\d W_s\right\rangle\cr
=&\frac{H-\ff 1 2}{\Gamma(\frac{3}{2}-H)}
\Bigg[\int_0^{t_0}\left\langle\ff {s^{\frac{1}{2}-H}\si^{-1}_s\zeta^\ve_s}{H-\ff 1 2},\d W_s\right\rangle\cr
&\qquad\qquad\quad+\int_0^{t_0}\left\langle s^{H-\frac{1}{2}}\si^{-1}_s\zeta^\ve_s\int_0^s\frac{s^{\frac{1}{2}-H}-r^{\frac{1}{2}-H}}{(s-r)^{\frac{1}{2}+H}}
\d r,\d W_s\right\rangle\cr
&\qquad\qquad\quad+\int_0^{t_0}\left\langle s^{H-\frac{1}{2}}\zeta^\ve_s\int_0^s\frac{\si^{-1}_s-\si^{-1}_r}{(s-r)^{\frac{1}{2}+H}}r^{\ff 1 2-H}\d r,\d W_s\right\rangle\cr
&\qquad\qquad\quad+\int_0^{t_0}\left\langle s^{H-\frac{1}{2}}\int_0^s\frac{\zeta^\ve_s-\zeta^\ve_r}{(s-r)^{\frac{1}{2}+H}}\si^{-1}_rr^{\ff 1 2-H}\d r,\d W_s\right\rangle\Bigg]\cr
=:&\frac{H-\ff 1 2}{\Gamma(\frac{3}{2}-H)}[J_1(t_0)+J_2(t_0)+J_3(t_0)+J_4(t_0)].
\end{align}
Note that by our hypotheses, \eqref{1PfTh(NBD)} and Lemma \ref{PD-GT}, it is readily verified that for any $r,s\in[0,t_0]$,
\beg{align}\label{5PfTh(NBD)}
\lim\limits_{\ve\downarrow0}\ff{\zeta^\ve_s}\ve=&\ff{\phi(X_0^\mu)+\Lambda_{t_0}}{t_0}+\na b_s(\cdot,P_s^*\mu)(X_s^\mu)\left(\ff{t_0-s}{t_0}\phi(X_0^\mu)-\ff s {t_0}\Lambda_{t_0}+\Lambda_s\right)\cr
&+\E[\langle D^Lb_s(x,\cdot)(P_s^*\mu)(X_s^\mu),\nabla_\phi X_s^\mu\rangle]|_{x=X_s^\mu}=:\Upsilon_{s,t_0}
\end{align}
and
 \beg{align}\label{8PfTh(NBD)}
\lim\limits_{\ve\downarrow0}\ff{\zeta^\ve_s-\zeta^\ve_r}\ve=&\na b_s(\cdot,P_s^*\mu)(X_s^\mu)\left(\ff{t_0-s}{t_0}\phi(X_0^\mu)-\ff s {t_0}\Lambda_{t_0}+\Lambda_s\right)\cr
&-\na b_r(\cdot,P_r^*\mu)(X_r^\mu)\left(\ff{t_0-r}{t_0}\phi(X_0^\mu)-\ff r {t_0}\Lambda_{t_0}+\Lambda_r\right)\cr
&+E[\langle D^Lb_s(x,\cdot)(P_s^*\mu)(X_s^\mu),\nabla_\phi X_s^\mu\rangle]|_{x=X_s^\mu}\cr
&-E[\langle D^Lb_r(y,\cdot)(P_r^*\mu)(X_r^\mu),\nabla_\phi X_r^\mu\rangle]|_{y=X_r^\mu}\cr
=&\Upsilon_{s,t_0}-\Upsilon_{r,t_0}.
\end{align}
Then, applying the dominated convergence theorem, we obtain that as $\ve\downarrow0$, $J_i(t_0)/\ve, i=1,\cdots,4$, converge to
\beg{align*}
&\int_0^{t_0}\left\langle\ff {s^{\frac{1}{2}-H}\si^{-1}_s\Upsilon_{s,t_0}}{H-\ff 1 2},\d W_s\right\rangle,\\
&\int_0^{t_0}\left\langle s^{H-\frac{1}{2}}\si^{-1}_s\Upsilon_{s,t_0}\int_0^s\frac{s^{\frac{1}{2}-H}-r^{\frac{1}{2}-H}}{(s-r)^{\frac{1}{2}+H}}
\d r,\d W_s\right\rangle,\\
&\int_0^{t_0}\left\langle s^{H-\frac{1}{2}}\Upsilon_{s,t_0}\int_0^s\frac{\si^{-1}_s-\si^{-1}_r}{(s-r)^{\frac{1}{2}+H}}r^{\ff 1 2-H}\d r,\d W_s\right\rangle
\end{align*}
and
\beg{align*}
\int_0^{t_0}\left\langle s^{H-\frac{1}{2}}\int_0^s\frac{\Upsilon_{s,t_0}-\Upsilon_{r,t_0}}{(s-r)^{\frac{1}{2}+H}}\si^{-1}_rr^{\ff 1 2-H}\d r,\d W_s\right\rangle.
\end{align*}
in $L^1(\P^{\ti H,0})$, respectively.
Consequently, combining these with \eqref{3PfTh(NBD)}, \eqref{4PfTh(NBD)} and \eqref{2PfTh(NBD)}, we conclude that
\beg{align}\label{6PfTh(NBD)}
&\lim\limits_{\ve\downarrow0}\ff{(P_{t_0}^{\ti H,0}f)(X_0^{\mu_{\ve,\phi}})-(P_{t_0}^{\ti H,0}f)(X_0^\mu)}{\ve}\cr
=&\E^{\ti H,0}\Bigg(f(X_{t_0}^\mu)\cdot
\frac{H-\ff 1 2}{\Gamma(\frac{3}{2}-H)}
\Bigg[\int_0^{t_0}\left\langle\ff {s^{\frac{1}{2}-H}\si^{-1}_s\Upsilon_{s,t_0}}{H-\ff 1 2},\d W_s\right\rangle\cr
&\qquad\qquad\quad\qquad\qquad\qquad+\int_0^{t_0}\left\langle s^{H-\frac{1}{2}}\si^{-1}_s\Upsilon_{s,t_0}\int_0^s\frac{s^{\frac{1}{2}-H}-r^{\frac{1}{2}-H}}{(s-r)^{\frac{1}{2}+H}}
\d r,\d W_s\right\rangle\cr
&\qquad\qquad\quad\qquad\qquad\qquad+\int_0^{t_0}\left\langle s^{H-\frac{1}{2}}\Upsilon_{s,t_0}\int_0^s\frac{\si^{-1}_s-\si^{-1}_r}{(s-r)^{\frac{1}{2}+H}}r^{\ff 1 2-H}\d r,\d W_s\right\rangle\cr
&\qquad\qquad\quad\qquad\qquad\qquad+\int_0^{t_0}\left\langle s^{H-\frac{1}{2}}\int_0^s\frac{\Upsilon_{s,t_0}-\Upsilon_{r,t_0}}{(s-r)^{\frac{1}{2}+H}}\si^{-1}_rr^{\ff 1 2-H}\d r,\d W_s\right\rangle\Bigg]\Bigg)\cr
=&\E^{\ti H,0}\left(f(X_{t_0}^\mu)\int_0^{t_0}\left\langle K_H^{-1}\left(\int_0^\cdot\si^{-1}_r\Upsilon_{r,t_0}\d r\right)(s),\d W_s\right\rangle\right).
\end{align}
Here we have used \eqref{InOp} and \eqref{FrDe} in the last relation.

Now, let $\sL_{Y|\P^{\ti H,0}}$ be the conditional distribution of a random variable $Y$ under $\P^{\ti H,0}$.
According to the Pinsker inequality, we have
\beg{align*}
\sup_{\|f\|_\infty\leq 1}\left|(P_{t_0}^{\ti H,0}f)(X_0^{\mu_{\ve,\phi}})-(P_{t_0}^{\ti H,0}f)(X_0^\mu)\right|^2
=&\sup_{\|f\|_\infty\leq 1}\left|\sL_{X_{t_0}^{\mu_{\ve,\phi}}|\P^{\ti H,0}}(f)-\sL_{X_{t_0}^\mu|\P^{\ti H,0}}(f)\right|^2\cr
\leq& 2\mathrm{Ent}\left(\sL_{X_{t_0}^{\mu_{\ve,\phi}}|\P^{\ti H,0}}|\sL_{X_{t_0}^\mu|\P^{\ti H,0}}\right).
\end{align*}
Then, using the equivalence between the log-Harnack inequality and the entropy-cost estimate (see Remark \ref{ReTh(Har)}),
it follows from \eqref{6PfTh(Har)} and Lemma \ref{GirsT} that
\beg{align*}
\sup_{\|f\|_\infty\leq 1}\left|(P_{t_0}^{\ti H,0}f)(X_0^{\mu_{\ve,\phi}})-(P_{t_0}^{\ti H,0}f)(X_0^\mu)\right|^2\leq 2\vartheta(H),
\end{align*}
where $\nu$ of $\vartheta(H)$ is replaced by $\mu_{\ve,\phi}$.\\
Consequently, this, along with the expression of $\vartheta(H)$ and Theorem \ref{WP}, leads to
\beg{align}\label{7PfTh(NBD)}
&\ff{|(P_{t_0}^{\ti H,0}f)(X_0^{\mu_{\ve,\phi}})-(P_{t_0}^{\ti H,0}f)(X_0^\mu)|}{\ve}\cr
\leq& C_{T,\kappa,\ti\kappa,H,\ti H}\|f\|_\infty\Bigg[\|\phi\|_{L^p(\mu)}+\bigg(\ff 1 {t_0^{H}}+\|\varrho^\mu_\cdot\|^{\be}_{\ti H-\varsigma_1}+
\|\varrho^{\mu_{\ve,\phi}}_\cdot\|^{\be}_{\ti H-\varsigma_2}+\psi^{\beta}(X_0,\varrho)\bigg)\tilde{\psi}(X_0,\varrho)\cr
&\qquad\qquad\qquad\quad+\left(1+|X_0^\mu|^{\be}+\|\varrho^\mu_\cdot\|_\infty^{\be}+\left\|\varrho^\mu_\cdot\right\|^{\be}_{\ti H-\varsigma_1}\right)(\|\phi\|_{L^p(\mu)}+\tilde{\psi}(X_0,\varrho))\cr
&\qquad\qquad\qquad\quad+\left(\int_0^{t_0}s^{2H-1}\left(\int_0^s\frac{|\varrho^{\mu_{\ve,\phi}}_s-\varrho^\mu_s-(\varrho^{\mu_{\ve,\phi}}_r-\varrho^\mu_r)|}{(s-r)^{\frac{1}{2}+H}}r^{\ff 1 2-H}\d r\right)^2\d s\right)^{\ff 1 2}\Bigg]
\end{align}
with $\tilde{\psi}(X_0,\varrho):=|\phi(X_0^\mu)|+\sup_{s\in[0,t_0]}|\varrho^\mu_s-\varrho^{\mu_{\ve,\phi}}_s|/\ve$.
Therefore, taking into account of \eqref{6PfTh(NBD)} and Remark \ref{Re(Est)}, applying the dominated convergence theorem yields
\beg{align}\label{9PfTh(NBD)}
D^L_\phi(P_{t_0}f)(\mu)=&\lim\limits_{\ve\downarrow0}\E\ff{(P_{t_0}^{\ti H,0}f)(X_0^{\mu_{\ve,\phi}})-(P_{t_0}^{\ti H,0}f)(X_0^\mu)}{\ve}\cr
=&\E\left(\lim\limits_{\ve\downarrow0}\ff{(P_{t_0}^{\ti H,0}f)(X_0^{\mu_{\ve,\phi}})-(P_{t_0}^{\ti H,0}f)(X_0^\mu)}{\ve}\right)\cr
=&\E\left(f(X_{t_0}^\mu)\int_0^{t_0}\left\langle K_H^{-1}\left(\int_0^\cdot\si^{-1}_r\Upsilon_{r,t_0}\d r\right)(s),\d W_s\right\rangle\right).
\end{align}

\textsl{The case $H\in(0,1/2)$.}
Using \eqref{-HOF} and \eqref{FrIn}, we first have
\beg{align*}
&\int_0^{t_0}\left\langle K_H^{-1}\left(\int_0^\cdot\si^{-1}_r\zeta_r^\ve\d r\right)(s),\d W_s\right\rangle
=\int_0^{t_0}\left\langle\frac{\si^{-1}s^{H-\ff 1 2}}{\Gamma(\frac{1}{2}-H)}\int_0^s\frac{r^{\ff 1 2-H}\zeta_r^\ve}{(s-r)^{\frac{1}{2}+H}}\d r,\d W_s\right\rangle.
\end{align*}
Reasoning as in \eqref{6PfTh(NBD)} and \eqref{7PfTh(NBD)} , it can be shown that
\beg{align}\label{10PfTh(NBD)}
&\lim\limits_{\ve\downarrow0}\ff{(P_{t_0}^{\ti H,0}f)(X_0^{\mu_{\ve,\phi}})-(P_{t_0}^{\ti H,0}f)(X_0^\mu)}{\ve}\cr
=&\E^{\ti H,0}\Bigg(f(X_{t_0}^\mu)\cdot
\int_0^{t_0}\left\langle\frac{\si^{-1}s^{H-\ff 1 2}}{\Gamma(\frac{1}{2}-H)}\int_0^s\frac{r^{\ff 1 2-H}\Upsilon_{r,t_0}}{(s-r)^{\frac{1}{2}+H}}\d r,\d W_s\right\rangle\cr
=&\E^{\ti H,0}\left(f(X_{t_0}^\mu)\int_0^{t_0}\left\langle K_H^{-1}\left(\int_0^\cdot\si^{-1}_r\Upsilon_{r,t_0}\d r\right)(s),\d W_s\right\rangle\right).
\end{align}
and
\beg{align*}
\ff{|(P_{t_0}^{\ti H,0}f)(X_0^{\mu_{\ve,\phi}})-(P_{t_0}^{\ti H,0}f)(X_0^\mu)|}{\ve}
\leq C_{T,\kappa,H,\ti H}\|f\|_\infty\left(\ff {\tilde{\psi}(X_0,\varrho)} {t_0^{H}}+\|\phi\|_{L^p(\mu)}\right).
\end{align*}
So, by Remark \ref{Re(Est)} and the dominated convergence theorem again, we deduce
\beg{align}\label{11PfTh(NBD)}
D^L_\phi(P_{t_0}f)(\mu)=&\lim\limits_{\ve\downarrow0}\E\ff{(P_{t_0}^{\ti H,0}f)(X_0^{\mu_{\ve,\phi}})-(P_{t_0}^{\ti H,0}f)(X_0^\mu)}{\ve}\cr
=&\E\left(f(X_{t_0}^\mu)\int_0^{t_0}\left\langle K_H^{-1}\left(\int_0^\cdot\si^{-1}_r\Upsilon_{r,t_0}\d r\right)(s),\d W_s\right\rangle\right).
\end{align}
Our proof is now finished.
\end{proof}

\begin{rem}\label{ReTh(NBD)}
(i) Due to \eqref{InOp} and \eqref{-HOF}, we can rewrite the term $K_H^{-1}\left(\int_0^\cdot\si^{-1}_r\Upsilon_{r,t}\d r\right)(s)$
on the right-hand side of \eqref{1Th(NBD)} as follows
\begin{equation}\begin{split}\nonumber
&K_H^{-1}\left(\int_0^\cdot\si^{-1}_r\Upsilon_{r,t}\d r\right)(s)\\
&=
\left\{
\begin{array}{ll}\vspace{0.3cm}
\frac{(H-\ff 1 2)s^{H-\ff 1 2}}{\Gamma(\frac{3}{2}-H)}
\Bigg[\ff {s^{1-2H}\si^{-1}_s\Upsilon_{s,t}}{H-\ff 1 2}
+\si^{-1}_s\Upsilon_{s,t}\int_0^s\frac{s^{\frac{1}{2}-H}-r^{\frac{1}{2}-H}}{(s-r)^{\frac{1}{2}+H}}\d r+\cr
\Upsilon_{s,t}\int_0^s\frac{(\si^{-1}_s-\si^{-1}_r)r^{\ff 1 2-H}}{(s-r)^{\frac{1}{2}+H}}\d r
+\int_0^s\frac{(\Upsilon_{s,t}-\Upsilon_{r,t})\si^{-1}_rr^{\ff 1 2-H}}{(s-r)^{\frac{1}{2}+H}}\d r\Bigg], H\in(\ff 1 2,1),\\
\frac{\si^{-1}s^{H-\ff 1 2}}{\Gamma(\frac{1}{2}-H)}
\int_0^s\frac{r^{\ff 1 2-H}\Upsilon_{r,t}}{(s-r)^{\frac{1}{2}+H}}\d r,\ \ \ \ \ \ \ \ \ \ \ \ \ \ \ \ \ \ \ \ \ \ \ \ \ \ \ \ \ \ \ \ \  H\in(0,\ff 1 2).
\end{array} \right.
\end{split}\end{equation}

(ii) Using Theorem \ref{Th(NBD)} and the H\"{o}lder inequality and following the similar argument as in Lemma \ref{GirsT},
we obtain
\beg{align*}
\|D^L(P_t f)(\mu)\|_{L^{p^*}_\mu}\le C_{T,\kappa,\ti\kappa,H,\ti H}\left(1+\ff {1} {t^H}\right)\left((P_t|f|^{p^*})(\mu)\right)^{\ff 1 {p^*}}
\end{align*}
with any $t\in(0,T], f\in\sB_b(\R^d)$ and $\mu\in\sP_p(\R^d)$,
where $C_{T,\kappa,\ti\kappa,H,\ti H}$ is a positive constant which is independent of $\ti\kappa$ when $H\in(0,1/2)$,
and $p\geq 2(1+\be)$ if $H\in(1/2,1)$ or $p\geq 2$ if $H\in(0,1/2)$.

\end{rem}

\subsection{The degenerate case}

Let $A$ and $B$ be two matrices of order $m\times m$ and $m\times l$, we now consider the following distribution dependent degenerate SDE:
\begin{equation}\label{(Deg)eq}
\begin{cases}
 \textnormal\d X^{(1)}_t=(AX^{(1)}_t+B X^{(2)}_t)\d t,\\
 \textnormal\d X^{(2)}_t=b_t(X_t,\sL_{X_t})\d t+\si_t\d B^{H}_t+\ti\si_t(\L_{X_t})\d\ti B^{\ti H}_t,
\end{cases}
\end{equation}
where $X_t=(X^{(1)}_t,X^{(2)}_t), b:[0,T]\times\R^{m+l}\times\sP_p(\R^{m+l})\ra\R^l, \si(t)$ is an invertible $l\times l$-matrix for every $t\in[0,T]$,
$\ti\si:[0,T]\times\scr P_p(\R^{m+l})\rightarrow\R^l\otimes\R^l$ are measurable.
It is worth pointing out that as in the Brownian motion case (see, e.g., \cite{BRW,RW}),
the above model is a distribution dependent stochastic Hamiltonian system with fractional noise.

\subsubsection{Log-Harnack inequality}

To establish the log-Harnack inequality, we let
\beg{align}\label{1DeTh(Har)}
U_t=\int_0^t\ff{s(t-s)}{t^2} \e^{-sA}BB^*\e^{-sA^*}\d s\geq\ell(t)\mathrm{I}_{m\times m}, \ t\in(0,T],
\end{align}
where $\ell\in C([0,T])$ satisfies $\ell(t)>0$ for any $t\in(0,T]$ and $\mathrm{I}_{m\times m}$ is the $m\times m$ identity matrix.
It is obvious that $U_t$ is invertible with $\|U_t^{-1}\|\leq 1/\ell(t)$ for every $t\in(0,T]$.
Then, our main result in the part can be stated in the following theorem.

\begin{thm}\label{DeTh(Har)}
Consider Eq. \eqref{(Deg)eq}. Assume \eqref{1DeTh(Har)} and
if one of the two following assumptions holds:
\beg{enumerate}
\item[(I)] $H\in(1/2,1)$, $b,\si,\ti\si$ satisfy \textsc{\textbf{(H1')}} and \textsc{\textbf{(H2)}} with $d=m+l$, and $p\geq 2(1+\be)$;
\item[(II)] $H\in(0,1/2), b,\ti\si$ satisfies \textsc{\textbf{(H1)}} with $d=m+l$, $\si_t$ does not depend on $t$ and $p\geq 2$.
\end{enumerate}
Then for any $t\in(0,T], \mu,\nu\in\sP_p(\R^{m+l})$ and $0<f\in\sB_b(\R^{m+l})$,
\beg{align*}
(P_t\log f)(\nu)\leq\log(P_t f)(\mu)+\chi(H),
\end{align*}
where
\begin{equation}\nonumber
\chi(H)=
\left\{
\begin{array}{ll}\vspace{0.3cm}
 C_{T,\kappa,\ti\kappa,H,\ti H}\left(1+\W_p(\mu,\nu)^{2\be}+\ff 1 {t^{2H}}+\ff 1 {\ell^2(t)}+\ff 1{t^{2H}\ell^2(t)}\right)\W_p(\mu,\nu)^2,  \ H\in(1/2,1),\\
 C_{T,\kappa,H,\ti H}\left(1+\ff 1{t^{2H}}+\ff 1 {\ell^2(t)}+\ff 1{t^{2H}\ell^2(t)}\right)\W_p(\mu,\nu)^2,\ \ \ \ \ \ \ \ \ \ \ \ \ \ \ \ \ \ \ H\in(0,1/2).
\end{array} \right.
\end{equation}
\end{thm}

\begin{proof}

For any $\mu,\nu\in\sP_p(\R^d)$, let $X_0^\mu$ and $X_0^\nu$ be $\sF_0$-measurable satisfying $\sL_{X_0^\mu}=\mu,\sL_{X_0^\nu}=\nu$ and
\beg{align}\label{1PfDT(Ha)}
\E|X_0^\mu-X_0^\nu|^p=\W_p(\mu,\nu)^p,
\end{align}
and let $X_t^\mu$ and $X_t^\nu$ solve respectively \eqref{(Deg)eq} with  $\sL_{X_0^\mu}=\mu$ and $\sL_{X_0^\nu}=\nu$, which  implies $\sL_{X_t^\mu}=P_t^*\mu$ and  $\sL_{X_t^\nu}=P_t^*\nu$.

Fix $t_0\in(0,T]$. We first introduce the following coupling DDSDE: for $t\in[0,t_0]$,
\begin{equation}\label{2PfDT(Ha)}
\begin{cases}
 \textnormal\d Y^{(1)}_t=(AY^{(1)}_t+B Y^{(2)}_t)\d t,\\
 \textnormal\d Y^{(2)}_t=(b_t(X_t,P_t^*\mu)+g'(t))\d t+\si_t\d B^{H}_t+\ti\si_t(P_t^*\nu)\d\ti B^{\ti H}_t,
\end{cases}
\end{equation}
with $Y_0=X_0^\nu$,
where the differentiable function $g:[0,t_0]\ra\R$ will be determinated below. \\
Combining \eqref{(Deg)eq} with \eqref{2PfDT(Ha)} yields that for each $t\in[0,t_0]$,
\begin{equation}\label{5PfDT(Ha)}
\begin{cases}
Y^{(1)}_t-X^{\mu,(1)}_t=\e^{tA}Z_0^{(1)}+\int_0^t\e^{(t-s)A}B(Z_0^{(2)}+g(s)-g(0)+\varrho_s^\nu-\varrho_s^\mu)\d s,\\
Y^{(2)}_t-X^{\mu,(2)}_t=Z_0^{(2)}+g(t)-g(0)+\varrho_t^\nu-\varrho_t^\mu,
\end{cases}
\end{equation}
where $Z_0=(Z_0^{(1)},Z_0^{(2)}):=Y_0-X^\mu_0=(Y^{(1)}_0-X^{\mu,(1)}_0,Y^{(2)}_0-X^{\mu,(2)}_0)$.\\
To construct a coupling $(X^\mu_t,Y_t)$ by change of measure for them such that $X^\mu_{t_0}=Y_{t_0}$, we take $g$ as follows:
\beg{align}\label{6PfDT(Ha)}
g(t)&=-\ff t {t_0}(Z_0^{(2)}+\varrho_{t_0}^\nu-\varrho_{t_0}^\mu)-\ff{t(t_0-t)}{t_0^2}B^*\e^{-tA^*}U_{t_0}^{-1}Z_0^{(1)}\cr
&-\ff{t(t_0-t)}{t_0^2}B^*\e^{-tA^*}U_{t_0}^{-1}\int_0^{t_0}\e^{-sA}B\left[\ff{t_0-s}{t_0}Z_0^{(2)}-\ff s {t_0}(\varrho_{t_0}^\nu-\varrho_{t_0}^\mu)+\varrho_s^\nu-\varrho_s^\mu\right]\d s.
\end{align}

Next, we rewrite \eqref{2PfDT(Ha)} as
\begin{equation}\label{3PfDT(Ha)}
\begin{cases}
 \textnormal\d Y^{(1)}_t=(AY^{(1)}_t+B Y^{(2)}_t)\d t,\\
 \textnormal\d Y^{(2)}_t=b_t(Y_t,P_t^*\nu)\d t+\si_t\d\hat{B}^{H}_t+\ti\si_t(P_t^*\nu)\d\ti B^{\ti H}_t,\ \ t\in[0,t_0],
\end{cases}
\end{equation}
where
\beg{align*}
\hat{B}^{H}_t:=B^{H}_t-\int_0^t\si^{-1}_s\hat{\zeta}_s\d s=\int_0^tK_H(t,s)\left(\d W_s-K_H^{-1}\left(\int_0^\cdot\si^{-1}_r\hat{\zeta}_r\d r\right)(s)\d s\right)
\end{align*}
with
\beg{align*}
\hat{\zeta}_s=b_s(Y_s,P_s^*\nu)-b_s(X_s,P_s^*\mu)-g'(s).
\end{align*}
Put
\beg{align*}
\hat{R}^{\ti H,0}:=\exp\left[\int_0^{t_0}\left\langle K_H^{-1}\left(\int_0^\cdot\si^{-1}_r\hat{\zeta}_r\d r\right)(s),\d W_s\right\rangle-\ff 1 2\int_0^{t_0}\left|K_H^{-1}\left(\int_0^\cdot\si^{-1}_r\hat{\zeta}_r\d r\right)(s)\right|^2\d s\right].
\end{align*}
By a direct calculation, we can have
\beg{align*}
|Y_t-X_t^\mu|\leq\ff {C_{T,\kappa}} {\ell(t_0)}\Big(|Z_0|+\sup_{s\in[0,t_0]}|\varrho^\mu_s-\varrho^\nu_s|\Big)
\end{align*}
and
\beg{align*}
|\hat{\zeta}_t|\leq C_{T,\kappa,\ti H}\left[\W_p(\mu,\nu)+\left(\ff 1{t_0}+\ff 1 {\ell(t_0)}+\ff 1 {t_0\ell(t_0)}\right)\Big(|Z_0|+\sup_{s\in[0,t_0]}|\varrho^\mu_s-\varrho^\nu_s|\Big)\right].
\end{align*}
Then, in the sprit of the proofs of Lemma \ref{GirsT} and Remark \ref{Re(GirsT)},
we conclude that $(\hat{B}^{H}_t)_{t\in[0,t_0]}$ is a $l$-dimensional fractional Brownian motion under the conditional probability $\hat{R}^{\ti H,0}\d\P^{\tilde{H},0}$.
and there holds
\beg{align*}
\E\left(\int_0^{t_0}\left|K_H^{-1}\left(\int_0^\cdot\si^{-1}_r\hat{\zeta}_r\d r\right)(s)\right|^2\d s\right)
\leq\chi(H)
\end{align*}
with
\begin{equation}\nonumber
\chi(H)=
\left\{
\begin{array}{ll}\vspace{0.3cm}
 C_{T,\kappa,\ti\kappa,H,\ti H}\left(1+\W_p(\mu,\nu)^{2\be}+\ff 1 {t_0^{2H}}+\ff 1 {\ell^2(t_0)}+\ff 1{t_0^{2H}\ell^2(t_0)}\right)\W_p(\mu,\nu)^2,  \ H\in(1/2,1),\\
 C_{T,\kappa,H,\ti H}\left(1+\ff 1{t_0^{2H}}+\ff 1 {\ell^2(t_0)}+\ff 1{t_0^{2H}\ell^2(t_0)}\right)\W_p(\mu,\nu)^2,\ \ \ \ \ \ \ \ \ \ \ \ \ \ \ \ \ \ \ \ H\in(0,1/2).
\end{array} \right.
\end{equation}

Now, let $\hat{Y}_t=Y_t-(0,\varrho_t^\nu)$ and then it is easy to see that $\hat{Y}_\cdot$ satisfies
\begin{equation}\label{4PfDT(Ha)}
\begin{cases}
 \textnormal\d \hat{Y}^{(1)}_t=(A\hat{Y}^{(1)}_t+B \hat{Y}^{(2)}_t+B\varrho_t^\nu)\d t,\\
 \textnormal\d \hat{Y}^{(2)}_t=b_t(\hat{Y}_t+(0,\varrho_t^\nu),P_t^*\nu)\d t+\si_t\d\hat{B}^{H}_t,\ \ t\in[0,t_0],\ \hat{Y}_0=Y_0.
\end{cases}
\end{equation}
Observe that $\hat{X}=X^\nu-(0,\varrho^\nu)$ solves SDE of the same form as \eqref{4PfDT(Ha)} with $\hat{B}^H$ replaced by $B^H$.
So, along the same lines as in \eqref{5PfTh(Har)}, \eqref{6PfTh(Har)} and \eqref{7PfTh(Har)}, we get the desired assertion.
\end{proof}

\subsubsection{Bismut formula}

In this part, we aim to establish the Bismut formula for the $L$-derivative of \eqref{(Deg)eq}.
For every $\mu\in\sP_p(\R^{m+l})$, let $X_0^\mu$ be $\sF_0$-measurable satisfying $\sL_{X_0^\mu}=\mu$,
and let $(X_t^\mu)_{t\in[0,T]}$ be the solution to \eqref{(Deg)eq} with initial value $X_0^\mu$.
For any $\ve\in[0,1]$ and $\phi\in L^p(\R^{m+l}\ra\R^{m+l},\mu)$, denote $X_t^{\mu_{\ve,\phi}}$ by the solution of \eqref{(Deg)eq} with $X_0^{\mu_{\varepsilon,\phi}}=(\mathrm{Id}+\varepsilon\phi)(X_0^\mu)$
and denote $P^*_t\mu_{\varepsilon,\phi}=\sL_{X_t^{\mu_{\varepsilon,\phi}}}$ for every $t\in[0,T]$.
We set for each $0\leq s<t\leq T$,
\beg{align*}
\hbar_{s,t}:=\left(\e^{sA}\phi^{(1)}(X_0^\mu)+\int_0^s\e^{(s-r)A}B\left(\phi^{(2)}(X_0^\mu)+\Xi_{t}(r)+\Lambda_r\right)\d r, \phi^{(2)}(X_0^\mu)+\Xi_{t}(s)+\Lambda_s\right),
\end{align*}
where
\beg{align*}
\Xi_{t}(s):=&-\ff s {t_0}(\phi^{(2)}(X_0^\mu)+\Lambda_{t})-\ff{s(t-s)}{t^2}B^*\e^{-sA^*}U_{t}^{-1}\phi^{(1)}(X_0^\mu)\cr
&-\ff{s(t-s)}{t^2}B^*\e^{-sA^*}U_{t}^{-1}\int_0^{t_0}\e^{-rA}B\left[\ff{t-r}{t_0}\phi^{(2)}(X_0^\mu)-\ff r {t_0}\Lambda_{t}+\Lambda_r\right]\d r.
\end{align*}

\begin{thm}\label{Th(BD)}
Consider Eq. \eqref{(Deg)eq}. Assume \eqref{1DeTh(Har)} and
if one of the two following assumptions holds:
\beg{enumerate}
\item[(I)] $H\in(1/2,1)$, $b,\si,\ti\si$ satisfy \textsc{\textbf{(H1')}}, \textsc{\textbf{(H2)}} and \textsc{\textbf{(H3)}};
\item[(II)] $H\in(0,1/2), b,\ti\si$ satisfies \textsc{\textbf{(H1')}}, \textsc{\textbf{(H3)}} with $d=m+l$, $\si_t$ does not depend on $t$,
\end{enumerate}
then for any $t\in(0,T],f\in\sB_b(\R^{m+l}),\phi\in L^p(\R^{m+l}\ra\R^{m+l},\mu)$ and $\mu\in\sP_p(\R^{m+l})$ with $p\geq 2(1+\be)$ if $H\in(1/2,1)$ or $p\geq 2$ if $H\in(0,1/2)$, $D^L_\phi(P_T f)(\mu)$ exists and satisfies
\beg{align*}
D^L_\phi(P_{t}f)(\mu)
=\E\left(f(X_{t}^\mu)\int_0^{t}\left\langle K_H^{-1}\left(\int_0^\cdot\si^{-1}_r\Theta_{r,t}\d r\right)(s),\d W_s\right\rangle\right),
\end{align*}
where $\Theta_{\cdot,\cdot}$ is defined as
\beg{align*}
\Theta_{s,t}=\na b_s(\cdot,P_s^*\mu)(X_s^\mu)\hbar_{s,t}+\E[\langle D^Lb_s(x,\cdot)(P_s^*\mu)(X_s^\mu),\nabla_\phi X_s^\mu\rangle]|_{x=X_s^\mu}
-(\Xi_{t})'(s).
\end{align*}
\end{thm}

\begin{proof}
Let $t_0\in(0,T]$ be fixed.
For $\ve\in(0,1]$, let $Y^\ve$ solve \eqref{2PfDT(Ha)} with $\nu=\mu_{\varepsilon,\phi}$ and $Y_0=Y_0^\ve=(\mathrm{Id}+\varepsilon\phi)(X_0^\mu)$.
Then, \eqref{5PfDT(Ha)} becomes
\begin{equation*}
\begin{cases}
Y^{\ve,(1)}_t-X^{\mu,(1)}_t=\ve\e^{tA}\phi^{(1)}(X_0^\mu)+\int_0^t\e^{(t-s)A}B(\ve\phi^{(2)}(X_0^\mu)+g(s)+\varrho_s^{\mu_{\varepsilon,\phi}}-\varrho_s^\mu)\d s,\\
Y^{\ve,(2)}_t-X^{\mu,(2)}_t=\ve\phi^{(2)}(X_0^\mu)+g(t)+\varrho_t^{\mu_{\varepsilon,\phi}}-\varrho_t^\mu.
\end{cases}
\end{equation*}
Here we recall that $g(0)=0$ due to \eqref{6PfDT(Ha)} in which $\nu$ and $(Z_0^{(1)},Z_0^{(2)})$ is replaced by $\mu_{\varepsilon,\phi}$ and $(\ve\phi^{(1)}(X_0^\mu),\ve\phi^{(2)}(X_0^\mu))$, respectively.
In particular, there holds $Y_{t_0}^\ve=X_{t_0}^\mu$.

Set
\beg{align*}
\hat{R}_\ve^{\ti H,0}:=\exp\left[\int_0^{t_0}\left\langle K_H^{-1}\left(\int_0^\cdot\si^{-1}_r\hat{\zeta^\ve_r}\d r\right)(s),\d W_s\right\rangle-\ff 1 2\int_0^{t_0}\left|K_H^{-1}\left(\int_0^\cdot\si^{-1}_r\hat{\zeta^\ve_r}\d r\right)(s)\right|^2\d s\right]
\end{align*}
with
\beg{align*}
\hat{\zeta}^\ve_s=b_s(Y_s^\ve,P_s^*{\mu_{\varepsilon,\phi}})-b_s(X_s^\mu,P_s^*\mu)-g'(s).
\end{align*}
Observe that as in \eqref{5PfTh(NBD)} and  \eqref{8PfTh(NBD)}, we obtain that for each $r,s\in[0,t_0]$,
\beg{align*}
\lim\limits_{\ve\downarrow0}\ff{\hat{\zeta}^\ve_s}\ve
=&\na b_s(\cdot,P_s^*\mu)(X_s^\mu)\hbar_{s,t_0}+\E[\langle D^Lb_s(x,\cdot)(P_s^*\mu)(X_s^\mu),\nabla_\phi X_s^\mu\rangle]|_{x=X_s^\mu}\cr
&-(\Xi_{t_0})'(s)=:\Theta_{s,t_0}
\end{align*}
and
\beg{align*}
\lim\limits_{\ve\downarrow0}\ff{\hat{\zeta}^\ve_s-\hat{\zeta}^\ve_r}\ve=&\na b_s(\cdot,P_s^*\mu)(X_s^\mu)\hbar_{s,t_0}-\na b_r(\cdot,P_r^*\mu)(X_r^\mu)\hbar_{r,t_0}\cr
&+\E[\langle D^Lb_s(x,\cdot)(P_s^*\mu)(X_s^\mu),\nabla_\phi X_s^\mu\rangle]|_{x=X_s^\mu}\cr
&-\E[\langle D^Lb_r(y,\cdot)(P_r^*\mu)(X_r^\mu),\nabla_\phi X_r^\mu\rangle]|_{y=X_r^\mu}\cr
&-\left[(\Xi_{t_0})'(s)-(\Xi_{t_0})'(r)\right]\cr
=&\Theta_{s,t_0}-\Theta_{r,t_0}.
\end{align*}
Then, resorting to the same techniques as in  \eqref{6PfTh(NBD)} and \eqref{9PfTh(NBD)} as well as \eqref{10PfTh(NBD)} and \eqref{11PfTh(NBD)}, we derive that for each $H\in(1/2,1)\cup(0,1/2)$,
\beg{align*}
D^L_\phi(P_{t_0}f)(\mu)
=\E\left(f(X_{t_0}^\mu)\int_0^{t_0}\left\langle K_H^{-1}\left(\int_0^\cdot\si^{-1}_r\Theta_{r,t_0}\d r\right)(s),\d W_s\right\rangle\right).
\end{align*}
\end{proof}

We conclude this part with a remark.

\begin{rem}\label{Re(DeTh)}
Similar to Remarks \ref{ReTh(Har)} and \ref{ReTh(NBD)}(ii), it follows from Theorems \ref{DeTh(Har)} and \ref{Th(BD)} that
the following entropy-cost and intrinsic derivative estimates
\beg{align*}
\mathrm{Ent}(P_t^*\nu|P_t^*\mu)\leq\chi(H)
\end{align*}
and
\beg{align*}
\|D^L(P_t f)(\mu)\|_{L^{p^*}_\mu}\le C_{T,\kappa,\ti\kappa,H,\ti H}\left(1+\ff 1 {t^H}+\ff 1 {\ell(t)}+\ff 1{t^{H}\ell(t)}\right)\left((P_t|f|^{p^*})(\mu)\right)^{\ff 1 {p^*}}
\end{align*}
hold for any $t\in(0,T], \mu,\nu\in\sP_p(\R^d)$ and $f\in\sB_b(\R^d)$,
where $C_{T,\kappa,\ti\kappa,H,\ti H}$ is a positive constant which is independent of $\ti\kappa$ when $H\in(0,1/2)$,
and $p\geq 2(1+\be)$ if $H\in(1/2,1)$ or $p\geq 2$ if $H\in(0,1/2)$.
In addititon, to guarantee \eqref{1DeTh(Har)} holds, one needs to impose some non-degeneracy condition on the matrix $B$.
For instance,  assume the following Kalman rank condition:
\beg{align}\label{1Re(DeTh}
\mathrm{Rank}[B,AB,\cdots,A^kB]=m
\end{align}
holds for some integer number $k\in[0,m-1]$ (in particular, if $k=0$, \eqref{1Re(DeTh} reduces to $\mathrm{Rank}[B]=m$),
then  \eqref{1DeTh(Har)} is satisfied with $\ell(t)=C(t\wedge1)^{2k+1}$
for positive constant $C$ (see, e.g., \cite[Proof of Theorem 4.2]{WZ13}).
\end{rem}

\end{document}